\documentclass{amsart}
\usepackage{fullpage}
\usepackage{graphicx}
\usepackage{amssymb}
\usepackage{amsmath}
\usepackage[margin=1in]{geometry}
\usepackage[toc,page]{appendix}
\newtheorem{thm}{Theorem}[section]

\newtheorem{prop}[thm]{Proposition}

\title{Complex orthogonal  geometric structures of dimension three}
\author{Mayra Mendez}
\begin{document}
\maketitle

\begin{abstract}
A \emph{complex orthogonal (geometric) structure} on a complex manifold is a geometric structure locally modelled on a non-degenerate quadric. One of the first examples of such a structure on a compact manifold of dimension three  was constructed by Guillot. In this paper, we show that the same manifold carries a family of uniformizable complex orthogonal (geometric) structures which includes Guillot's structure; here, a structure is said to be uniformizable if it is a quotient of an invariant open set of a quadric by a Kleinian group. We also construct a family of uniformizable complex (geometric) projective structures on a related compact complex manifold of dimension three.
\end{abstract}

\section{Introduction} \label{s1}
A \emph{(classical) Kleinian group} $\Gamma$ is a discrete subgroup of the group of M\"{o}bius transformations which acts properly discontinuously on some non-empty invariant open set of the Riemann sphere. It is well-known that every classical Kleinian group $\Gamma$ splits the Riemann sphere into two sets: the limit set and the discontinuity region; the dynamics of the group $\Gamma$ is concentrated on the limit set, while the geometry lives in the discontinuity region. In fact, if the group acts freely on the discontinuity region $\Omega$, then the quotient $\Gamma\setminus\Omega$ inherits the local structure of the Riemann sphere: it is a Riemann surface such that the projection $\Omega\to \Gamma\setminus\Omega$ is a local biholomorphism; thus,  one may say that, in the classical setting, there is a strong relationship between the geometry and the dynamics of a Kleinian group.


The M\"{o}bius transformations can be characterized either as the conformal automorphisms of the Riemann sphere which preserve the orientation or as the biholomorphisms of the complex projective space of dimension one or, finally, as the projective transformations of the complex projective plane which preserve a one-dimensional non-degenerate quadric (conic). Accordingly, there are, at least, three natural generalizations of the classical Kleinian groups to higher dimensions:
\begin{itemize}
\item A \emph{conformal Kleinian group} is a discrete subgroup of the group $\mathrm{Conf}^{+}(\mathbb{S}^{n})$ of conformal orientation-preserving automorphisms of the $n$-dimensional sphere $\mathbb{S}^{n}$ that acts properly discontinuously on a non-empty invariant open set of $\mathbb{S}^{n}$.
\item A \emph{complex  Kleinian group} is a discrete subgroup of the group $\mathrm{PSL}(n+1,\mathbb{C})$ of projective transformations of the $n$-dimensional complex proyective space $\mathbb{CP}^{n}$ that acts properly discontinuously on a non-empty invariant open set of $\mathbb{CP}^{n}$.
\item A  \emph{complex orthogonal Kleinian group} is a discrete subgroup of the group  $\mathrm{PO}(n+1,\mathbb{C})$ of projective transformations which preserve the  $n$-dimensional non-degenerate quadric $\mathrm{Q}_{n}$ that acts properly discontinuously on a non-empty invariant open set of this quadric.
\end{itemize}
The geometric structure (see Goldman \cite{GoldmanGE}) determined by the quotient of a conformal Kleinian group, a complex Kleinian group or a complex orthogonal Kleinian group is called a \emph{uniformizable conformal structure}, a \emph{uniformizable complex projective structure} or a \emph{uniformizable complex orthogonal structure}, respectively.

Of these three kinds of groups, conformal Kleinian groups are the best-understood so far; a complete survey can be found in \cite{Kapovich}. Much work has also been done on higher-dimensional complex Kleinian groups; some of the first examples were given by Kato \cite{Kato}, Larusson  \cite{larusson}, Nori \cite{Nori} and Seade and Verjovsky \cite{Alberto}. Complex orthogonal Kleinian groups are the least studied at the moment; one of the first examples on dimension three was constructed by Guillot in \cite[p.\ 224, 225]{GuillotD}. The first result of this paper is a family of uniformizable complex Kleinian groups which includes Guillot's example.

\begin{thm} \label{res1}
Let $\Gamma \subset \mathrm{SL}(2,\mathbb{C})$ be a  torsion free, finitely-generated, (classical) Kleinian group with domain of discontinuity $\Omega$ in $\mathbb{CP}^{1}$.  For the quadric 
$$\mathrm{Q}_{3}:= \{ [z_{1}:z_{2}:z_{3}:z_{4}:z_{5}]:  z_{1}z_{5}-z_{2}z_{4}-z_{3}^{2}=0\}$$
and the embedding
\begin{eqnarray} \label{enc}
\mathrm{SL}(2,\mathbb{C}) & \hookrightarrow & \mathrm{Q}_{3}, \nonumber \\
\left( \begin{array}{cc}
            a & b \\
            c & d
           \end{array}
    \right) & \mapsto & [a:b:1:c:d],  \nonumber
\end{eqnarray}
consider the (unique) extension of the action of $\mathrm{SL}(2,\mathbb{C}) \times \mathrm{SL}(2,\mathbb{C})$ on $\mathrm{SL}(2,\mathbb{C})$, which sends $\big( (g,h),x \big)$ to $gxh^{-1},$ to $\mathrm{Q}_{3}$.

Then,  $\mathrm{Q}_{3}-\mathrm{SL}(2,\mathbb{C})$ is biholomorphic to $\mathbb{CP}^{1} \times \mathbb{CP}^{1}$ and, for every group homomorphism
 $u:\Gamma \to \mathrm{SL}(2,\mathbb{C})$,  such that  
 \begin{equation}\label{gamma}
 \Gamma_{u}:= \Big\{ \big( \gamma,u(\gamma) \big): \gamma\in\Gamma \Big\}
 \end{equation} 
 acts properly discontinuously on $\mathrm{SL}(2,\mathrm{C})$; then,  $ \Gamma_{u}$ acts properly discontinuously on 
 \begin{equation} \label{ugamma}
 U_{\Gamma}:=\mathrm{SL}(2,\mathbb{C}) \cup \big( \Omega\times\mathbb{CP}^{1} \big).
 \end{equation}
  Moreover, if $\Gamma\setminus\Omega$ is compact, then $U_{\Gamma}$ is maximal.
\end{thm}
While this paper was in preparation, examples similar to those of this Theorem were obtained, independently using other techniques, by Gu\'eritaud, Guichard, Kassel and Wienhard (see
Theorem 4.1 and Observation 4.3 of \cite{artfanny}). However, our treatment of the subject is different; in principle, there may exist cases of our Theorem with $\Gamma$ not convex-cocompact and these would have no counterpart in the work of those authors. At the moment we do not know of any such examples. \\

The examples constructed by Guillot correspond to the quotient $\Gamma_{I} \setminus U_{\Gamma}$ of  this Theorem, where $I$ is the constant morphism and $\Gamma$ is a convex-cocompact Kleinian group. \\

The geometric study of the complex and the complex orthogonal Kleinian groups is complicated by the fact that there is no  good way to define an analogue of the discontinuity region in these cases. This makes the examples of the corresponding uniformizable  structures all the more valuable. One of the first examples of  a compact  manifold with a uniformizable complex orthogonal  structure of dimension three was given by Guillot in \cite[p.\ 224, 225]{GuillotD} as the  quotient of his example of complex Kleinian groups of dimension three.
The main result of the present paper says that Guillot's example is part of a family of uniformizable complex orthogonal  structures on the same manifold:

\begin{thm} \label{res2}
Let $\Gamma \subset \mathrm{SL}(2,\mathbb{C})$ be a torsion-free, convex-cocompact, (classical) Kleinian group with domain of discontinuity $\Omega$ in $\mathbb{CP}^{1}$. Consider the action  of $\mathrm{SL}(2,\mathbb{C}) \times \mathrm{SL}(2,\mathbb{C})$ on the three-dimensional non-degenerate quadric $\mathrm{Q}_{3}$, defined in Theorem \ref{res1}, the open set $U_{\Gamma} \subset \mathrm{Q}_{3}$, defined in ($\ref{ugamma}$), and for each group morphism  $u:\Gamma\to \mathrm{SL}(2,\mathbb{C})$, the group $\Gamma_{u}$, defined in ($\ref{gamma}$).
 
  Then, for each group morphism  $u:\Gamma\to \mathrm{SL}(2,\mathbb{C})$, sufficiently close to the constant morphism, $U_{\Gamma}$ is a maximal open set where $\Gamma_{u}$ acts properly discontinuously. Also, for all homomorphisms $u$, the quotients $\Gamma_{u} \setminus U_{\Gamma}$ are compact and diffeomorphic to each other.
\end{thm}
The examples constructed by A. Guillot correspond to the quotient of $\Gamma_{I} \setminus U_{\Gamma}$ of  this Theorem, where $I$ is the constant morphism. We call \emph{the Guillot manifold}, the quotient manifold (both, differentiable and complex) and \emph{the Guillot structure}, the complex orthogonal  structure determined by it.

\medskip

We will also construct uniformizable complex projective  structures on a related complex manifold of dimension three.

\begin{thm} \label{corart}
Let $\Gamma \subset \mathrm{SL}(2,\mathbb{C})$ be a torsion free, convex-cocompact, classical Kleinian group with domain of discontinuity $\Omega$ in $\mathbb{CP}^{1}$. Consider  $\mathbb{CP}^{3}$ as the projectivization of the space of $2 \times 2$ complex matrices and the action of $\mathrm{SL}(2,\mathbb{C}) \times \mathrm{SL}(2,\mathbb{C})$  on it that sends $\big( (g,h),[x] \big)$ to $gxh^{-1}$.  Then, there exists an open set $V_{\Gamma} \subset \mathbb{CP}^{3}$, such that, for each group homomorphism  $u:\Gamma\to \mathrm{SL}(2,\mathbb{C})$  sufficiently close to the constant morphism, $V_{\Gamma}$ is a maximal open set where
 $\Gamma_{u}$, defined in (\ref{gamma}),
acts properly discontinuously. Also, for all  $u$, the quotients $\Gamma_{u} \setminus V_{\Gamma}$ are compact and diffeomorphic to each other.
\end{thm}
The complex manifold of this Theorem  and  the uniformizable complex projective  structure induced by $\Gamma_{I} \setminus U_{\Gamma}$, where $I$ is the constant morphism, were also  found by Guillot.

The spaces of homomorphisms  from 
$\Gamma$ to $\mathrm{SL}(2,\mathbb{C}) \times \mathrm{SL}(2,\mathbb{C})$ of Theorems \ref{res2} and \ref{corart}, are considered with the compact-open topology. As we will see, the groups $\Gamma_{u}$ of these theorems are embedded as subgroups into $\mathrm{PO}(5,\mathbb{C})$ and $\mathrm{PO}(4,\mathbb{C})$.

If $u$ is close to the constant morphism, the homomorphism $\gamma \mapsto \big( \gamma, u(\gamma) \big)$ is close to $\gamma \mapsto \big( \gamma, I \big)$. Then, each group $\Gamma_{u}$ of Theorem \ref{res2}  determines an uniformizable complex orthogonal  structure on the Guillot manifold, which is close to the Guillot structure. If $u$ and $v$ are close to the constant morphism, the geometric structures determined by $\Gamma_{u}$ and $\Gamma_{v}$ coincide if and only if $u$ and $v$ are conjugate. The same phenomenon occurs
in the context of Theorem \ref{corart}. 

 \medskip

The proofs of Theorems~\ref{res1} and ~\ref{res2} go as follows. First, we will consider a torsion-free, finitely generated, (classical) Kleinian group $\Gamma\subset \mathrm{SL}(2,\mathbb{C})$ with domain of discontinuity $\Omega$ in $\mathbb{CP}^{1}$ and $u:\Gamma\to\mathrm{SL}(2,\mathbb{C})$ a group morphism. 
Then, we will recall that if we consider the intersection of $\mathrm{Q}_{3}$ with the projectivization of the hyperplane in $\mathbb{C}^{5}$ defined by $z_{3}=0$ we get the two-dimensional quadric $\mathrm{Q}_{2}$. We will also recall that there exists a $\big( \mathrm{SL}(2,\mathbb{C}) \times\mathrm{SL}(2,\mathbb{C}) \big)$-equivariant biholomorphism from $\mathrm{Q}_{2}$ to $\mathbb{CP}^{1} \times \mathbb{CP}^{1}$. Next, we  will consider the action of $\Gamma_{u}$ on $\mathbb{CP}^{1} \times \mathbb{CP}^{1}$ defined by   $$\Big( \big(\gamma,u(\gamma) \big) (x,y) \Big) \mapsto \big(\gamma(x),u(\gamma^{-1})(y) \big),$$ where $\big(\gamma,u(\gamma) \big) \in \Gamma_{u}$ and $(x,y) \in \mathbb{CP}^{1} \times \mathbb{CP}^{1}$. Since this action is properly discontinuous in the first coordinate of $\Omega\times\mathbb{CP}^{1}$, it follows that $\Gamma_{u}$ acts properly discontinuously and \emph{uniformly} on  $\Omega\times\mathbb{CP}^{1}$ (by the uniformity of the action, we mean that, for every compact set, there is a bound on length of the $\Gamma_{u}$-translates of this compact set that intersect it, and the bound is independent of $u$).

Then, we will develop some of the ideas and techniques of Frances in \cite{Frances} in order to study the dynamics of the compact sets of $\mathrm{Q}_{3}$ for divergent sequences of $\Gamma_{u}$. In particular, we will prove that if $\Gamma_{u}$, defined in (\ref{gamma}), acts properly discontinuously on $\mathrm{SL}(2,\mathbb{C})$, then it acts properly discontinuously on 
$U_{\Gamma}$, defined in (\ref{ugamma}).
Moreover, if $\Gamma\setminus\Omega$ is compact then,  $U_{\Gamma}$ is maximal. This way, Theorem \ref{res1} will be proved. \\

In order to prove Theorem \ref{res2}, we will consider the group $\Gamma$ to be convex-cocompact. Then,  we will  generalize Lema 2.1 of Ghys \cite[p. 119]{GhysD} to prove that $\Gamma_{u}$ acts properly discontinuously and \emph{uniformly} on $\mathrm{SL}(2,\mathbb{C})$, for all $u$ sufficiently close to the constant morphism. So, the hypothesis of Theorem \ref{res1} are valid and then, for all $u$ sufficiently close to the constant morphism, $\Gamma_{u}$ acts properly discontinuously on $U_{\Gamma}$.  

Then, we will continue developing the ideas  and techniques of Frances in order to prove that  $\Gamma_{u}$ acts uniformly on $U_{\Gamma}$. So, there exists an open neighborhood $\mathcal{V}$ of the constant morphism $I$ such that $\Gamma$ acts properly discontinuously on $\mathcal{V} \times U_{\Gamma}$. If $\mathcal{V}$ is a manifold, then this means that 
\begin{eqnarray*}
\Gamma\setminus \big(\mathcal{V} \times U_{\Gamma} \big)  & \to & \mathcal{V},  \\
\ [\nu,x]  & \mapsto & \nu, 
\end{eqnarray*}
where $\nu\in\mathcal{V}, x \in U_{\Gamma}$,  is a locally trivial fibration; this proves the theorem. If $\mathcal{V}$ is not a manifold, we will  consider a resolution of singularities $r:X \to \mathcal{V}$ of a neighborhood  $\mathcal{V}$ of the constant morphism to construct a locally trivial fibration over $X$ whose fibers are the quotients  $\Gamma_{u} \setminus U_{\Gamma}$ and the Theorem will be proved in the general case.

\medskip

In order to prove Theorem~\ref{corart}, we will show that $\Gamma_{u}$ is a subgroup of $\mathrm{PO}(4,\mathbb{C})$ and that there exists a $\big(\mathrm{SL}(2,\mathbb{C}) \times \mathrm{SL}(2,\mathbb{C}) \big)$-equivariant, continuous, proper and open map from $\mathrm{Q}_{3}$ to $\mathbb{CP}^{3}$. We will push forward the set $U_{\Gamma}$ of Theorem \ref{res2} to get the set $V_{\Gamma}$ of Theorem \ref{corart}.

\medskip

Section~\ref{s2} is dedicated to the geometry of $\mathrm{Q}_{3}$. In Section \ref{ss3}, we consider the action of $\Gamma_{u}$ on $\mathrm{Q}_{2}$ and, on $\mathrm{SL}(2,\mathbb{C})$, for the homomorphisms $u:\Gamma\to\mathrm{SL}(2,\mathbb{C})$  sufficiently close to the constant morphism. In Section \ref{s4}, we study the dynamics of the accumulation points for the orbits of compact sets of $\mathrm{Q}_{3}$ for divergent sequences of $\Gamma_{u}$; we prove that if $\Gamma_{u}$ acts properly discontinuously on $\mathrm{SL}(2,\mathbb{C})$, then it acts properly discontinuously on  $U_{\Gamma}$ and if $\Gamma\setminus \Omega$ is compact, then $U_{\Gamma}$ is maximal.  
In Section \ref{s5}, we prove that for $u$ sufficiently close to the constant morphism,  all the quotients  $\Gamma_{u} \setminus U_{\Gamma}$ are compact and diffeomorphic to each other. Finally, we construct the  $\big(\mathrm{SL}(2,\mathrm{C}) \times \mathrm{SL}(2,\mathrm{C}) \big) $-equivariant, continuous, open and proper map from $\mathrm{Q}_{3}$ to $\mathbb{CP}^{3}$ and push forward the complex orthogonal Kleinian group $\Gamma_{u}$ to get a complex Kleinian group. \\

The author would like to thank Adolfo Guillot for all his help and support.


\section{The geometry of the quadric} \label{s2}

In this Section, we study the geometry of the non-degenerate quadric $\mathrm{Q}_{3}$ of dimension three and its group $\mathrm{PO}(5,\mathbb{C})$ of transformations. We will consider the non-degenerate quadric $\mathrm{Q}_{2}$ of dimension two obtained by intersecting $\mathrm{Q}_{3}$ with the projectivization of the hyperplane $z_{3}=0$. We will  recall that the orthogonal  group $\mathrm{O}(4,\mathbb{C})$  is a  subgroup of $\mathrm{PO}(5,\mathbb{C})$ which preserves $\mathrm{Q}_{2}$ and that the group $\mathrm{SO}(4,\mathbb{C})$ of orthogonal  matrices of determinant one is isomorphic to  $\big( \mathrm{SL}(2,\mathbb{C}) \times \mathrm{SL}(2,\mathbb{C}) \big) / \big\{ (I,I), (-I,-I) \big\}$.

We will also define two important kinds of subsets of $\mathrm{Q}_{3}$ and study their geometry; namely, light geodesics and light cones. In Section \ref{s4}, we will see that these sets appear naturally as the sets of accumulation points of the orbits of compact sets of $\mathrm{Q}_{3}$ under discrete subgroups of $\mathrm{SO}(4,\mathbb{C})$.

\subsection{The quadric and its automorphism group} \label{aut}
The reader can consult  Guillot \cite{GuillotD} and M\'endez  \cite{MayraTM, MayraTD} for further discussion of this Section.

\medskip

The non-degenerate quadratic form
\begin{equation} \label{defq} q(z_{1},z_{2},z_{3},z_{4},z_{5}):=z_{1}z_{5}-z_{2}z_{4}-z_{3}^{2}
\end{equation}
 on  $\mathbb{C}^{5}$ defines the non-degenerate quadratic form
 $$q^{*}(z_{1},z_{2},z_{4},z_{5}):=q(z_{1},z_{2},0,z_{4},z_{5})$$
 on  $\mathbb{C}^{4}$. The groups $\mathrm{O}(4,\mathbb{C})$ and  $\mathrm{O}(5,\mathbb{C})$  consist of the matrices which preserve $q^{*}$ and $q$, respectively; the group $\mathrm{SO}(4,\mathbb{C})$ is the subgroup of $\mathrm{O}(4,\mathbb{C})$ which contains the matrices of determinant one. We say that two matrices in $\mathrm{O}(n,\mathbb{C})$  ($n=4,5$) are equivalent if one of them is a nonzero $\mathbb{C}^{*}$-multiple of the other and $\mathrm{PO}(n,\mathbb{C})$  is the set of equivalence classes.

Consider the  quadric
\begin{equation} \label{defc}
C_{4}:= \big\{ (z_{1},z_{2},z_{3},z_{4},z_{5}) \in\mathbb{C}^{5}-\{0\}:q(z_{1},z_{2},z_{3},z_{4},z_{5})=0 \big\}
\end{equation}
in $\mathbb{C}^{5}$  and the non-degenerate quadric 
\begin{equation} \label{defq3}
\mathrm{Q}_{3}:= \big\{ [z_{1}:z_{2}:z_{3}:z_{4}:z_{5}]\in\mathbb{CP}^{4}: \ q(z_{1},z_{2},z_{3},z_{4},z_{5})=0 \big\},
\end{equation}
in $\mathbb{CP}^{4}$. Then, $\mathrm{PO}(5,\mathbb{C})$ is the group of projective transformations which preserves $\mathrm{Q}_{3}$.

Let $\mathrm{H}$ be the hyperplane in $\mathbb{C}^{5}$ given by  $z_{3}=0$; denote by $\pi$ the projection
\begin{eqnarray}
\pi: \mathbb{C}^{5}-\{0\}  & \to & \mathbb{CP}^{4},  \nonumber\\
\big( z_{1},z_{2},z_{3},z_{4},z_{5} \big) & \mapsto & [z_{1}: z_{2}: z_{3}: z_{4}: z_{5}] 
\end{eqnarray}
and let
$$\mathrm{Q}_{2}:=\mathrm{Q}_{3} \cap \pi(\mathrm{H}).$$


Since the composition of this embedding and the projection of $\mathrm{O}(5,\mathbb{C})$ onto $\mathrm{PO}(5,\mathbb{C})$ defines a holomorphic monomorphism $\phi$ from $\mathrm{O}(4,\mathbb{C})$ to  $\mathrm{PO}(5,\mathbb{C})$. We also use the notation $\mathrm{O}(4,\mathbb{C})$ for its image; similarly, we write  $\mathrm{SO}(4,\mathbb{C})$ for the image of this group in $\mathrm{PO}(5,\mathbb{C})$.%
\\

The group $\mathrm{O}(4,\mathbb{C})$ is isomorphic to the subgroup of $\mathrm{PO}(5,\mathbb{C})$ that preserves the projection $\pi(\mathrm{H})$ of the hyperplane $\mathrm{H}$ and the projection $\pi(e_{3})$ of the vector $e_{3}$. \\

Recall the embedding (\ref{enc}) considered in the Introduction; as the  group $\mathrm{O}(4,\mathbb{C})$ preserves $\mathrm{Q}_{2}$, therefore,  
\begin{equation} \label{theta}
\Theta:=\mathrm{Q}_{3}-\mathrm{Q}_{2}
\end{equation}
is biholomorphic to $\mathrm{SL}(2,\mathbb{C})$ and
the action of
 $\mathrm{SL}(2,\mathbb{C}) \times  \mathrm{SL}(2,\mathbb{C})$ on $\mathrm{SL}(2,\mathbb{C})$ given by
 \begin{equation} \label{acsl}
 \big( (f,g), x \big) \mapsto fxg^{-1},
 \end{equation}
 where $f,g,x \in \mathrm{SL}(2,\mathbb{C})$,
 defines a holomorphic action of $\mathrm{SL}(2,\mathbb{C})\times  \mathrm{SL}(2,\mathbb{C})$ on $\Theta$. This action
 extends in a unique way to a (non-faithful) action on $\mathrm{Q}_{3}$, so it defines an holomorphic homomorphism $\psi$ from $\mathrm{SL}(2,\mathbb{C})\times \mathrm{SL}(2,\mathbb{C})$ to $\mathrm{PO}(5,\mathbb{C})$ whose image is contained in $\mathrm{SO}(4,\mathbb{C})$ and whose kernel is $\big\{ (I,I), (-I,-I)\big\}$. As the image of this homomorphism is a connected subgroup of the connected group $\mathrm{SO}(4,\mathbb{C})$  (see \cite[p. 82]{Goodman}) and both of them are of the same dimension, this homomorphism is surjective. Therefore, $\psi$ induces a biholomorphic isomorphism from $\big( \mathrm{SL}(2,\mathbb{C})\times  \mathrm{SL}(2,\mathbb{C}) \big) / \big\{ (I,I), (-I,-I) \big\}$  to $\mathrm{SO}(4, \mathbb{C})$. \\

The quadric $\mathrm{Q}_{2}$ is biholomorphic to $\mathbb{CP}^{1} \times \mathbb{CP}^{1}$. The function
\begin{eqnarray} \label{de2}
\mathrm{Q}_{2} & \to & \mathbb{CP}^{1}\times\mathbb{CP}^{1}, \nonumber\\
\ [z_{1}:z_{2}:0:z_{4}:z_{5}] & \mapsto & \Bigg( \Bigg[ Im \left( \begin{array}{cc}
            z_{1} & z_{2} \\
            z_{4} & z_{5}
           \end{array}
    \right) \Bigg], \Bigg[ Ker \left( \begin{array}{cc}
            z_{1} & z_{2} \\
            z_{4} & z_{5}
           \end{array}
    \right) \Bigg] \Bigg),
\end{eqnarray}
is a biholomorphism which is $\mathrm{SO}(4,\mathbb{C})$-equivariant with respect to the restriction of the action of $\mathrm{SO}(4,\mathbb{C})$ on $\mathrm{Q}_{3}$ and the action
\begin{eqnarray} \label{defacxc}
\mathrm{SO}(4,\mathbb{C}) \times \big( \mathbb{CP}^{1} \times \mathbb{CP}^{1} \big) &
\to &   \mathbb{CP}^{1} \times \mathbb{CP}^{1}, \nonumber\\
 \Big(\big[ (g,h) \big], (x,y) \Big) &\mapsto &\big( g(x),h(y) \big).
 \end{eqnarray}

\subsection{Light geodesics and light cones}
The results of this Section are analogous to those of the real case given by Frances in \cite{Frances}.

Consider the bilinear form  $b$ associated to the non-degenerate quadratic form $q$ defined in (\ref{defq}).
For each subspace $W$ of $\mathbb{C}^{5}$,
 $W^{\perp}$ is the set of vectors $v \in \mathbb{C}^{5}$ such that $b(v,w)=0$ for all $w \in W$.
 A vector subspace  $W$ of $\mathbb{C}^{5}$ is called  \emph{isotropic} if  $q(w)=0$ for all $w \in W$.
There exist isotropic $\mathbb{C}$-planes, for example $\langle e_{1}, e_{2} \rangle$, where $e_{1}, \dots, e_{5}$ is the canonical base of $\mathbb{C}^{5}$.

For every subspace $W$ of $\mathbb{C}^{5}$, $\mathrm{dim}(W)+\mathrm{dim}(W^{\perp}) = \mathrm{dim}(\mathbb{C}^{5})$.
If $W$ is isotropic, by the polarization identity, we know that $W \subset W^{\perp}$; thus, there are no isotropic subspaces of $\mathbb{C}^{5}$ of dimension three or four.

The projectivization of an isotropic $\mathbb{C}$-plane is called a \emph{light geodesic}. The group $\mathrm{PO}(5,\mathbb{C})$ sends light geodesics to light geodesics.

If $p \in \mathrm{Q}_{3}$, the union of all the light geodesics which contain $p$ is called the \emph{light cone} of $p$  and is denoted by $C(p)$. We have that if $\widetilde{p}$ is any point in $\mathbb{C}^{5}$ such that $\pi(\widetilde{p})=p$, then
$$C(p)=\pi( \widetilde{p}^{\perp} \cap C^{4}),
$$
where $C_{4}$ was defined in (\ref{defc}).

Let us consider the following equivalence relation in $C(p)-\{p\}$: we say that $x,y \in C(p)-\{p\}$ are \emph{equivalent} if they belong to the same light geodesic which contain $p$.
Let us denote by $\widetilde{C}(p)$ the space of all light geodesics which contain $p$, that is,  the space of all equivalence classes of $C(p)-\{p\}$. The group $\mathrm{PO}(5,\mathbb{C})$ sends light cones to light cones.

\begin{prop} \label{lg}
The space of all light geodesics of $\mathrm{Q}_{3}$ which pass through a given point is a $\mathbb{CP}^{1}$. Those geodesics which are contained in  $\mathrm{Q}_{2}$ are of the form $\{z\} \times\mathbb{CP}^{1}$ and $\mathbb{CP}^{1} \times \{w\}$, where $z,w \in\mathbb{CP}^{1}$.
\end{prop}

\begin{proof}

Recall that the non-degenerate quadric
$$\mathrm{Q}_{1}:=\big\{[0:z_{2}:z_{3}:z_{4}:0] \in\mathbb{CP}^{2}: z_{2}z_{4}+z_{3}^{2}=0 \big\}
$$is biholomorphic to $\mathbb{CP}^{1}$.

We have that
$$C \big( \pi(e_{1}) \big)=\pi(e_{1}^{\perp} \cap C^{4})= \big\{ [z_{1}:z_{2}:z_{3}:z_{4}:0] \in\mathbb{CP}^{4}:z_{2}z_{4}+z_{3}^{2}=0 \big\}.$$

The preimages of the map
\begin{eqnarray*}
p:C  \big(\pi(e_{1}) \big)-\big\{\pi(e_{1}) \big\}= \underset{[z_{2}:z_{3}:z_{4}] \in \mathrm{Q}_{1}}{\bigcup} \pi \big( \mathbb{C} \times \{z_{2} \} \times \{z_{3}\} \times \{z_{4} \} \times \{0\} \big) & \to & \mathrm{Q}_{1}, \\
\ [z_{1}:z_{2}:z_{3}:z_{4}:0] & \mapsto & [0:z_{2}:z_{3}:z_{4}:0]
\end{eqnarray*}
are the light geodesics which contain $\pi(e_{1})$, without the point $\pi(e_{1})$.
Therefore, the space of light geodesics which contains $\pi(e_{1})$ is a complex manifold biholomorphic to $\mathbb{CP}^{1}$.

In a similar fashion,  for all $g \in \mathrm{PO}(5,\mathbb{C})$, the map
 \begin{eqnarray*}
p_{g}:  C \Big(\pi \big(g (e_{1}) \big)  \bigg)-\big\{\pi(g(e_{1})) \big\}= \underset{[z_{2}:z_{3}:z_{4}] \in \mathrm{Q}_{1}}{\bigcup} g \Big( \pi \big( \mathbb{C} \times \{z_{2} \} \times \{z_{3}\} \times \{z_{4} \} \times \{0\} \big) \Big) \to g ( \mathrm{Q}_{1} ), && \\
\ g \big( [z_{1}:z_{2}:z_{3}:z_{4}:0] \big)  \mapsto  g \big( [0:z_{2}:z_{3}:z_{4}:0] \big). &&
\end{eqnarray*}
gives the space of light geodesics which contain $\pi \big(g (e_{1}) \big)$ a structure of a complex manifold biholomorphic to $\mathbb{CP}^{1}$. Also, $g$ induces a biholomorphism $\bar{g}$ from $\widetilde{C} \big( \pi(e_{1}) \big)$ to $\widetilde{C} \Big( \pi \big( g (e_{1}) \big) \Big)$. \\

Recall that we defined $\mathrm{H}$ to be the hyperplane $z_{3}=0$. The light geodesics contained in $\mathrm{Q}_{2}$ are
$$C \big( \pi(e_{1}) \big) \cap \pi(\mathrm{H})= \big\{ [z_{1}:z_{2}:0:z_{4}:0] \in\mathbb{CP}^{3}:z_{2}z_{4}=0 \big\}.
$$

Let us define the map
\begin{eqnarray*}
q: \Big( C \big( \pi(e_{1}) \big) \cap \pi(\mathrm{H}) \Big)- \big\{ \pi(e_{1}) \big\} & \to &  \big\{ [0:1], [1:0] \big\}, \\
\ [z_{1}:z_{2}:z_{4}:0] & \mapsto & [z_{2}:z_{4}].
\end{eqnarray*}

Then, $$q^{-1} \big( [0:1] \big) \cup \big\{ \pi(e_{1}) \big\}= \pi \big(\langle e_{1}, e_{4} \rangle \big)$$ and $$q^{-1} \big( [1:0] \big) \cup \big\{ \pi(e_{1}) \big\}= \pi \big( \langle e_{1}, e_{2} \rangle \big)$$ are the light geodesics which contain $\pi(e_{1})$.
Under the biholomorphism (\ref{de2}), these light geodesics are $\big\{ [1:0] \big\} \times\mathbb{CP}^{1}$ and $\mathbb{CP}^{1} \times \big\{ [0:1] \big\}$. \\

There are two natural foliations on $\mathbb{CP}^{1} \times\mathbb{CP}^{1}$, namely the vertical and the horizontal foliations, whose leaves are the sets of the form $\{z\} \times\mathbb{CP}^{1}, \ z \in\mathbb{CP}^{1}$ and $\mathbb{CP}^{1} \times \{w\}, \ w \in \mathbb{CP}^{1}$, respectively. As the action of $\mathrm{SO}(4,\mathbb{C})$ is transitive on $\mathbb{CP}^{1} \times\mathbb{CP}^{1}$ and  preserves the space of light geodesics, then, the leaves of these foliations are all the light geodesics contained in $\mathbb{CP}^{1} \times \mathbb{CP}^{1}$. We will call these leaves \emph{vertical} and \emph{horizontal} light geodesics, respectively. If two light geodesics are both horizontal (or vertical), then we say that they are \emph{parallel}.
\end{proof}


\section{Some deformations of (classical) Kleinian groups} \label{ss3}
In the last Section, we saw that there is an holomorphic epimorphism $\psi: \mathrm{SL}(2,\mathbb{C})  \times \mathrm{SL}(2,\mathbb{C})  \to \mathrm{SO}(4,\mathbb{C})$. 
In this Section, we will recall that 
 (classical) Kleinian groups of $\mathrm{SL}(2,\mathbb{C}) \cong \mathrm{SL}(2,\mathbb{C}) \times \{ I \} $ inject in $\mathrm{SO}(4,\mathbb{C})$, via this epimorphism. We will also recall a result of Guillot in  \cite[p.\ 224, 225]{GuillotD} which says that these groups are in fact, orthogonal Kleinian groups. We will also study some deformations of them inside $\mathrm{SO}(4,\mathbb{C})$ and their geometry in the quadric $\mathrm{Q}_{3}$ and in $\Theta=\mathrm{Q}_{3}-\mathrm{Q}_{2}$. \\

Consider the projection $\mathcal{P}: \mathrm{SL}(2,\mathbb{C}) \to  \mathrm{PSL}(2,\mathbb{C}), A \mapsto [A], A \in \mathrm{SL}(2,\mathbb{C})$.
 A \emph{lift of an element} $g$ in $\mathrm{PSL}(2,\mathbb{C})$ to $\mathrm{SL}(2,\mathbb{C})$ is  an element $A$ of $\mathrm{SL}(2,\mathbb{C})$ such that $\mathcal{P}(A)=g$.
 A \emph{lift of a subgroup} $\widetilde{\Gamma}$ of $\mathrm{PSL}(2,\mathbb{C})$ to $\mathrm{SL}(2,\mathbb{C})$ is a monomorphism $j:\widetilde{\Gamma} \to \mathrm{SL}(2,\mathbb{C})$ such that $\mathcal{P} \circ j$ is the identity.
A \emph{(classical) Kleinian group of $\mathrm{SL}(2,\mathbb{C})$} is the image of a lift of a (classical) Kleinian group of $\mathrm{PSL}(2,\mathbb{C})$ to $\mathrm{SL}(2,\mathbb{C})$. \\

 Consider a torsion-free (classical) Kleinian group $\Gamma$ of $\mathrm{SL}(2,\mathbb{C})$ (I. Kra in \cite{Kra} proved that they exist). It follows that $\Gamma$ acts on $\mathbb{CP}^{1}$
 \begin{eqnarray*}
   \Gamma \times \mathbb{CP}^{1} & \to & \mathbb{CP}^{1} \\
   (\gamma,z) & \to & \mathcal{P}(\gamma)(z).
 \end{eqnarray*}
 
Also, if $\Gamma$ is a lift of the (classical) Kleinian group $\widetilde{\Gamma} \subset \mathrm{PSL}(2,\mathbb{C})$ and $\Omega$ is the domain of discontinuity of $\widetilde{\Gamma}$ in $\mathbb{CP}^{1}$, then $\Gamma$ acts properly discontinuously on  $\Omega$ and all the properties of $\widetilde{\Gamma}$ are still valid for $\Gamma$.

 By definition $\mathcal{P}(-I)=Id$, where $I$ is the identity matrix and $Id$ the identity  M\"{o}bius transformation. So, if there exists an element $a \in \widetilde{\Gamma}$ such that $j(a)=-I$, as $\mathcal{P} \circ j$ is the identity, then $a=Id$; however, this is a contradiction, since $j$ is an isomorphism and, hence, $j(Id)=I$. Then $-I \notin\Gamma$. By similar arguments it follows that, if $A \in \Gamma$, then $-A \notin \Gamma$.

The composition of the inclusion of $\mathrm{SL}(2,\mathbb{C})$ into $\mathrm{SL}(2,\mathbb{C}) \times \mathrm{SL}(2,\mathbb{C})$, defined by $g \mapsto (g,I)$,
 and $\psi$, defined
in the Subsection \ref{aut}, defines a holomorphic monomorphism from $\mathrm{SL}(2,\mathbb{C})$ to $\mathrm{PO}(5,\mathbb{C})$; hence, $\mathrm{SL}(2,\mathbb{C})$ can be considered as a subgroup of $\mathrm{PO}(5,\mathbb{C})$. \\

Consider a   (classical) Kleinian group $\Gamma\subset \mathrm{SL}(2,\mathbb{C})$  and a group morphism $u: \Gamma\to \mathrm{SL}(2,\mathbb{C})$. Recall the definition of the group $\Gamma_{u}$, given in (\ref{gamma}); the kernel of the homomorphism $\psi$, given in Section \ref{aut}, and that $-I \notin \Gamma$. Then $\psi$, restricted to $\Gamma_{u}$, is also a monomorphism, and so, $\Gamma_{u}$ can be considered as a subgroup of $\mathrm{PO}(4,\mathbb{C})$. That is, the group $\Gamma_{u}$ is a deformation of the (classical) Kleinian group $\Gamma$ inside $\mathrm{SO}(4,\mathbb{C})$.
In fact, $\Gamma_{u}$ is torsion free.

One of the first examples of orthogonal Kleinian groups of dimension three was given by A. Guillot in \cite[pp. 224, 225]{GuillotD}. Guillot proved that, if $\Gamma$ is a torsion free, (classical) Kleinian group of $\mathrm{SL}(2,\mathbb{C})$, then $\Gamma$ is a complex orthogonal Kleinian group of dimension three (by means of the homomorphism $\psi$, defined in Section \ref{aut}). In particular, if $\Omega$ denotes the domain of discontinuity of $\Gamma$ in $\mathbb{CP}^{1}$, then $\Gamma$ acts  properly discontinuously on $\Theta \cup \big( \Omega\times\mathbb{CP}^{1} \big)$ (see Section \ref{s2}). Guillot also proved that if, in addition, $\Gamma$ is a convex-cocompact classical Kleinian group of $\mathrm{SL}(2,\mathbb{C})$, then $\Gamma$ acts cocompactly on $\Theta \cup \big( \Omega\times\mathbb{CP}^{1} \big)$. As we discussed in the Introduction of this paper, we call this quotient the Guillot manifold and the Guillot structure the complex orthogonal structure determined by it. \\

Recall that a (classical) Kleinian group $\Gamma$ with domain of discontinuity $\Omega$ in $\mathbb{CP}^{1}$ is \emph{convex-cocompact} if $\Gamma\setminus (\mathbb{H}^{3} \cup \Omega)$ is compact.  Classical Fuchsian groups which define compact surfaces of genus $g$  for $g \ge 1$ and (classical) Schottky groups are examples of convex-cocompact groups. A good reference for convex-cocompact Kleinian groups is \cite{Benoist}.

\begin{prop} \label{p1}
Let $\Gamma \subset \mathrm{SL}(2,\mathbb{C})$ be a torsion-free, convex-cocompact, (classical) Kleinian group with domain of discontinuity $\Omega$ in $\mathbb{CP}^{1}$. If  $u:\Gamma\to \mathrm{SL}(2,\mathbb{C})$ is a group morphism sufficiently close to the constant morphism, then, $\Gamma_{u}$, which was defined in (\ref{gamma}), acts properly discontinuously and uniformly on $\mathrm{SL}(2,\mathbb{C})$.

\end{prop}
\noindent\textbf{Remark 1:} 
In Theorem 1.3 of \cite{Fanny}, Kassel  proved that $\Gamma_{u}$ acts properly discontinuously on $\mathrm{SL}(2,\mathbb{C})$.
\noindent\textbf{Remark 2:} 
The proof of this Proposition is based on a modification of Lemma 1.2 of Ghys in \cite{GhysD}.

\begin{proof} 
We will show that there exists an open neighborhood $\mathcal{V}$ of the constant morphism such that
\begin{eqnarray} \label{f1}
  \Gamma \times \Big( \mathcal{V} \times \mathrm{SL}(2,\mathbb{C}) \Big) & \to & \mathcal{V} \times \mathrm{SL}(2,\mathbb{C}), \nonumber\\
  \Big( \gamma,\big( u,y \big) \Big) & \mapsto &  \Big( u, \big( \gamma,  u(\gamma) \big) y \Big)
\end{eqnarray}
is properly discontinuous.

In Lemma 1.2 of \cite{GhysD}, Ghys  showed the same statement but for discrete and cocompact groups, that is, he proved that if $\Gamma$ is discrete and cocompact, then, $\Gamma_{u}$ acts properly discontinuously and uniformly on $\mathrm{SL}(2,\mathbb{C})$.

In his proof, Ghys considered three different topologies on $\Gamma$: The word metric (we denote by $l(\gamma)$  the lenght of $\gamma$), the restriction of the metric induced by any right-invariant Riemmannian metric $d$ on $\mathrm{SL}(2,\mathbb{C})$ and the restriction of any Euclidean norm $||\cdot||$ on $\mathbb{C}^{4}$. Then, he used the fact that $\Gamma \setminus \mathrm{SL}(2,\mathbb{C})$ is compact to apply  the \v{S}varc-Milnor Lemma and concluded that   $l(\cdot)$  is bounded by a function which depends  linearly on $d(I,\cdot)$. On the other hand, it also happens that, in $\mathrm{SL}(2,\mathbb{C})$, $d(I,\cdot)$ is always bounded by a function which depends logarithmically on $||\cdot||$. Therefore, $l(\cdot)$ is bounded by a function which depends   logarithmically on $||\cdot||$.

Ghys considered a compact set $K \subset \mathrm{SL}(2,\mathbb{C})$, an element $\gamma$ in $\Gamma$ and a group morphism $u: \Gamma \to \mathrm{SL}(2,\mathbb{C})$, such that the $\big(\gamma,u(\gamma) \big)$-translate of $K$ intersects $K$. By routine analysis, it follows that $||\gamma||$ is bounded by a constant which depends exponentially on $l(\gamma)$.
By this and by the last paragraph, $l(\gamma)$ is bounded by a constant which depends linearly on $l(\gamma)$. It happens that, if the group morphism $u$ is sufficiently close to the constant morphism, then,  we can get a constant  upper bound on $l(\gamma)$ and this constant  does not depend on the homomorphism $u$. So, if $u$ is sufficiently close to the constant morphism, only a finite number of $\Gamma_{u}$-translates of $K$  intersect $K$ and the set of these translates does not depend on $u$.  \\

We will present a modification of Ghys' Lemma. We will consider a convex-cocompact (classical) Kleinian group $\Gamma$ of $\mathrm{SL}(2,\mathbb{C})$ and first prove that 
 $\mathrm{SL}(2,\mathbb{C})$ and $\mathbb{H}^{3}$ are quasi-isometric with respect to any left-invariant Riemannian metric in $\mathrm{SL}(2,\mathbb{C})$ and any word metric in $\Gamma$. Then, we will use that $\Gamma$ is convex-cocompact; in particular, there exists a $\Gamma$-invariant convex set of $\mathbb{H}^{3}$, such that, its orbit space is compact, in order to apply the \v{S}varc-Milnor Lemma and prove that a similar upper bound of $l(\cdot)$ (as a linear function of $d(I,\cdot)$) is valid. \\

We will first construct a specific left-invariant Riemannian metric in $\mathrm{SL}(2,\mathbb{C})$ and prove that $\mathrm{SL}(2,\mathbb{C})$ and $\mathbb{H}^{3}$ are quasi-isometric with respect to this specific left-invariant Riemannian metric in $\mathrm{SL}(2,\mathbb{C})$ and any word metric in $\Gamma$. As any two  left-invariant Riemannian metric in $\mathrm{SL}(2,\mathbb{C})$ are quasi-isometric, this will imply that  $\mathrm{SL}(2,\mathbb{C})$ and $\mathbb{H}^{3}$ are quasi-isometric with respect to any left-invariant Riemannian metric in $\mathrm{SL}(2,\mathbb{C})$ and any word metric in $\Gamma$. \\

It can be proved that
\begin{eqnarray} \label{defrho}
\rho:\mathrm{SL}(2,\mathbb{C}) & \to & \mathbb{H}^{3}, \nonumber\\
g & \mapsto & g(i)
\end{eqnarray}
is a trivial fibration with fiber $\mathrm{SU}(2)$ (for further discussion see \cite[pp. 37-38]{MayraTD}) and, hence, $T_{I}\mathrm{SL}(2,\mathbb{C})=T_{I}\mathbb{H}^{3} \oplus T_{I}\mathrm{SU}(2)$. Let us denote by $g^{1}$,  the hyperbolic metric of $\mathbb{H}^{3}$ and, by $g_{I}^{1}$, the inner product in $T_{I}\mathbb{H}^{3}$ determined by $g^{1}$ at the identity. Consider an arbitrary inner product $g_{I}^{2}$ in $T_{I}\mathrm{SU}(2)$. Then 
$$g_{I}: T_{I}\mathrm{SL}(2,\mathbb{C}) \times T_{I}\mathrm{SL}(2,\mathbb{C}) \to  \mathbb{R},$$ defined by
$$
g_{I}(v+u,l+p):=g_{I}^{1}(v,l)+g_{I}^{2}(u,p), \qquad v,l \in T_{I}\mathbb{H}^{3}, u,p \in T_{I}SU(2)
$$
is an inner product in $T_{I}\mathrm{SL}(2,\mathbb{C})$. Let us consider the corresponding left invariant Riemannian metric in $\mathrm{SL}(2,\mathbb{C})$:
$$g_{s}(u,v):=g_{I} \big( D_{s}L_{s^{-1}}(u), D_{s}L_{s^{-1}}(v) \big), \qquad u,v \in T_{s}\mathrm{SL}(2,\mathbb{C}),
$$
where $L_{s^{-1}}: G \to G$ is the left multiplication by $s^{-1}$ and $D_{s}$ denotes the differential of a function at the point $s$.

After some computation, we prove that $\rho$ is a quasi-isometry, with respect to the Riemannian metric $g_{I}$ in $\mathrm{SL}(2,\mathbb{C})$ and the hyperbolic metric $g^{1}$ in $\mathbb{H}^{3}$ (for further discussion see \cite[pp. 42-44]{MayraTD}). \\

 Let us consider the convex hull $C(\Gamma)$ of the limit set $\Lambda$ of $\Gamma$. Then $C(\Gamma) \cap \mathbb{H}^{3}$ is a closed, convex and $\Gamma$-invariant subset of $\mathbb{H}^{3}$.
 By definition of a convex-cocompact Kleinian group, we know that the convex-core
$$C(M):= \Gamma \setminus \big( C(\Gamma) \cap \mathbb{H}^{3} \big),
$$
of $\Gamma \setminus \mathbb{H}^{3},$ is compact. Without loss of generality, we can suppose that $(0,0,1) \in C(M)$ and so, by the \v{S}varc Milnor Lemma, $\Gamma  \to  C(M), \ \gamma \mapsto  g \big(  (0,0,1) \big)$ is a quasi-isometry. Then, by the latter and, as $\rho$ is a quasi-isometry, we get that there exists constants $B,C>0$, such that,
$$B^{-1} l(\gamma)-C \le d(\gamma,Id) \le B l(\gamma)+C.
$$

We can now follow the same arguments of E. Ghys to prove that (\ref{f1}) is properly discontinuous.
\end{proof}

\begin{prop} \label{l0}
If $\Gamma \subset \mathrm{SL}(2,\mathbb{C})$ is a classical Kleinian  group of  $\mathrm{SL}(2,\mathbb{C})$ with domain of discontinuity $\Omega$ in $\mathbb{CP}^{1}$. Then $\Gamma_{u}$ acts properly discontinuously and uniformly on $\Omega\times\mathbb{CP}^{1}$, for all $u \in \mathrm{Hom} \big(\Gamma,\mathrm{SL}(2,\mathbb{C}) \big)$.
\end{prop}
\begin{proof} We will show that 
\begin{eqnarray} \label{upoc}
\Gamma \times \Big( \mathrm{Hom} \big(\Gamma,\mathrm{SL}(2,\mathbb{C}) \big) \times  \big( \Omega \times\mathbb{CP}^{1} \big) \Big) & \to &
\Big( \mathrm{Hom} \big(\Gamma,\mathrm{SL}(2,\mathbb{C}) \big) \times  \big( \Omega \times\mathbb{CP}^{1} \big) \Big), \nonumber\\
\Big( \gamma, \big(u,(z,w)  \big) \Big) & \mapsto & \Big( u, \big( \gamma z, u(\gamma^{-1}) w \big) \Big),
\end{eqnarray}
is properly discontinuous. \\

As the projection
\begin{eqnarray*}
   \Omega \times\mathbb{CP}^{1}  & \to & \Omega, \nonumber\\
(x,w)  & \mapsto & x
\end{eqnarray*}
is continuous and  $\Gamma$-equivariant and $\Gamma$ acts properly discontinuously on $\Omega$, it follows that $\Gamma$ acts properly discontinuously on $\Omega\times\mathbb{CP}^{1}$.  As $\Gamma$ acts properly discontinuously on the second factor of (\ref{upoc}), then,
we get the Proposition.
\end{proof}


\section{Divergent sequences and limit sets}  \label{s4}
 We begin this Section giving a formula to compute the accumulation points of the orbits of compact sets in $\mathbb{CP}^{m}$ for divergent sequences of
$\mathrm{GL}(m+1,\mathbb{C})$. This formula gives a qualitative expression to the idea of Frances (see Propositions 3, 4 and 5 of \cite{Frances}) that the dynamics of the compact sets, of a divergent sequence and of a compact permutation of this divergent sequence, are equal. \\

Let us consider a continuous action $G \times X \to X$ of a topological group $G$ on a locally compact metric space $X$.

We say that $x \in X$ is \emph{dynamically related} to $y \in X$ if there exists a convergent sequence $(x_{n})$ of $X$  and a divergent sequence $(g_{n})$ of  $G$, such that, $x=\lim x_{n}$, $(g_{n}x_{n})$ is convergent and $y=\lim g_{n}x_{n}$. This relation is not necessarily reflexive nor transitive, but it is symmetric, and this allows us to say, without any ambiguity, that two points are dynamically related. If $(g_{n})$ is divergent, let us denote by $D_{(g_{n})}(x)$ the set of points $y \in X$, such that there exists a sequence $(x_{k})$ of $X$ which converges to $x$, together with a subsequence $(g_{n_{k}})$ of $(g_{n})$, such that $(g_{n_{k}}x_{k})$ is convergent and $y=\lim g_{n_{k}}x_{k}$. Let us define
$$D_{(g_{n})}(U):=\underset{x \in U}{\bigcup}D_{(g_{n})}(x),
$$
and
$$D_{G}(x)=\underset{(g_{n}) \subset G}{\bigcup}D_{(g_{n})}(x), \qquad D_{G}(U)=\underset{(g_{n}) \subset G}{\bigcup}D_{(g_{n})}(U),
$$
where the union is taken over all divergent sequences of $G$ and $U$ is an open set of $X$. That is, $D_{G}(x)$ is the set of all points of $X$ which are dynamically related to $x$.
We say that $y \in X$ is an \emph{accumulation point of the orbit of the compact set} $K \subset X$ if every neighborhood of $y$ intersects infinitely many $G$-translates of $K$.

If $U \subset X$ is an open set and $G$ is a discrete group,
the set $D_{G}(U)$  equals the set of all the accumulation points of the orbits of $G$ of all the compact sets of $U$.

\begin{prop} \label{paokak}
Let $m \in\mathbb{N}$, suppose that $u,\widetilde{u} \in \mathrm{GL}(m+1,\mathbb{C})$, for all $n \in\mathbb{N}$, $g_{n}, u_{n},
\widetilde{u}_{n}$ also belong to $\mathrm{GL}(m+1,\mathbb{C})$,  $u_{n} \to u$, $\widetilde{u}_{n} \to \widetilde{u}$, $(g_{n})$ is a divergent sequence of $\mathrm{GL}(m+1,\mathbb{C})$  and  $U$ is an open set of $\mathbb{CP}^{m}$, then,
\begin{eqnarray*}
D_{(\widetilde{u}_{n} g_{n}u_{n})} (U)=\widetilde{u} \Big( D_{(g_{n})} \big(u(U) \big) \Big).
\end{eqnarray*}
\end{prop}

\begin{proof}
We will prove this Proposition for $m=4$, the general proof is similar.
It is well-known that, if a sequence of $4 \times 4$ matrices converges in the  usual topology, then, it converges locally uniformly on $\mathbb{CP}^{4}$; then, $u_{n} \to u$ and $\widetilde{u}_{n} \to \widetilde{u}$, locally uniformly on $\mathbb{CP}^{4}$.
We will first prove that
\begin{equation} \label{paokak1}
D_{(g_{n}u_{n})} \big( u^{-1}(U) \big)  = D_{(g_{n})}(U).
\end{equation}

Suppose that $z \in D_{(g_{n}u_{n})} \big( u^{-1}(U) \big)$; then, we can assume that there exists a convergent sequence $(y_{n})$ of $u^{-1}(U)$ such that $(g_{n} u_{n} (y_{n}))$ is convergent and $z = \lim g_{n} u_{n} (y_{n})$. Let $y:=\lim y_{n} \in u^{-1}(U)$; in order to prove that $z \in D_{(g_{n})}(U)$, we will show that  $u_{n}( y_{n}) \to u(y)$.

Consider any $\epsilon >0$. Since $u_{n} \to u$ locally uniformly, $y_{n} \to y$ and $u$ is continuous, there exists $M \in\mathbb{N}$, such that for all $n \ge M$,
$$d \big(u_{n}(y_{n}),u(y) \big) \le d \big( u_{n}(y_{n}),u(y_{n}) \big)+d \big( u(y_{n}),u(y) \big)< \epsilon.
$$
Then, $u_{n}(y_{n}) \to u(y)$.

Now, suppose that $w \in D_{(g_{n})}(U)$. We can then assume that there exists a convergent sequence $(z_{n})$ of $U$, such that, $( g_{n}z_{n})$ is convergent and $w = \lim g_{n}z_{n}$. Let $z:=\lim z_{n} \in U$; in order to prove that $w \in D_{(g_{n}u_{n})} \big( u^{-1}(U) \big)$, we will show that
$$u_{n}^{-1}(z_{n}) \to u^{-1}(z).$$

Consider any $\epsilon >0$. Since $u_{n}^{-1} \to u^{-1}$
locally uniformly, $z_{n} \to z$ and $u^{-1}$ is continuous, there exists $N \in\mathbb{N}$, such that, for all  $n \ge N$,
$$d \big( u_{n}^{-1}(z_{n}), u^{-1}(z)\big) \le d \big( u_{n}^{-1}(z_{n}),u^{-1}(z_{n}) \big)+d \big(u^{-1}(z_{n}),u^{-1}(z) \big)< \epsilon.
$$

Then, $u_{n}^{-1}(z_{n}) \to u^{-1}(z)$. This proves (\ref{paokak1}). \\

Now, we will prove
\begin{equation} \label{paokak2}
D_{(\widetilde{u}_{n} g_{n})} (U)=\widetilde{u} \big( D_{(g_{n})}(U) \big).
\end{equation}

Consider $z \in D_{(\widetilde{u}_{n} g_{n})} (U)$. We can assume that there exists a convergent sequence $(x_{n})$ of $U$, such that, $\big( \widetilde{u}_{n} g_{n} (x_{n}) \big)$ is convergent and $z=\lim \widetilde{u}_{n} g_{n} (x_{n})$; if not, take a subsequence. In order to show that $z \in \widetilde{u} \big( D_{(g_{n})}(U) \big)$, we will show that there exists  $y \in D_{(g_{n})}(U)$, such that, $z=\widetilde{u}(y)$.

Consider any $\epsilon>0$ and let $z_{n}:=\widetilde{u}_{n} g_{n} (x_{n})$. As $\widetilde{u}_{n}^{-1} \to \widetilde{u}^{-1}$ locally uniformly in $X$, $\widetilde{u}^{-1}$ is continuous and $z_{n} \to z$, there exists $N \in\mathbb{N}$, such that, for all
 $n \ge N$,
$$
d \big( \widetilde{u}^{-1}_{n}(z_{n}), \widetilde{u}^{-1}(z) \big) \le
d \big( \widetilde{u}^{-1}_{n}(z_{n}),\widetilde{u}^{-1}(z_{n}) \big)+d \big( \widetilde{u}^{-1}(z_{n}),\widetilde{u}^{-1}(z) \big)<\epsilon.
$$

So, if $y:= \widetilde{u}^{-1}(z)$, then $g_{n}(x_{n}) \to y$.

Now, suppose that $z \in \widetilde{u} \big( D_{(g_{n})}(U) \big)$, then we can assume that there exists a convergence sequence $(x_{n})$ of $U$, such that $(g_{n}(x_{n}))$ is convergent and if $y:= \lim g_{n} (x_{n})$; then $z= \widetilde{u}(y)$. In order to prove that $z \in D_{(\widetilde{u}_{n} g_{n})}(U)$, we will show that
$$z= \lim \widetilde{u}_{n} g_{n} (x_{n}).$$

Consider any $\epsilon>0$;
since $\widetilde{u}$ is continuous, $\widetilde{u}_{n} \to \widetilde{u}$ locally uniformly and  $g_{n}(x_{n}) \to y$, there exists $N \in\mathbb{N}$, such that, for all $n \ge N$
$$d \big( z, \widetilde{u}_{n} g_{n} (x_{n}) \big) \le d \big( z,\widetilde{u}(g_{n}x_{n}) \big)+d \big( \widetilde{u}(g_{n}x_{n}),\widetilde{u}_{n}(g_{n}x_{n}
) \big) < \epsilon.$$

Then $\widetilde{u}_{n} g_{n} (x_{n}) \to z$. This proves (\ref{paokak2}). 

Finally, (\ref{paokak1}) and (\ref{paokak2}) prove the Proposition.
\end{proof} 

Now, let us consider a torsion-free, finitely generated, (classical) Kleinian group $\Gamma$ and a group morphism $u:\Gamma\to\mathrm{SL}(2,\mathbb{C})$   sufficiently close to the constant morphism. Recall the definition of the group $\Gamma_{u}$ given in (\ref{gamma}). Recall also that, in Section \ref{ss3}, we saw that $\Gamma_{u}$ is a subgroup of  $\mathrm{PO}(5,\mathbb{C})$ and a deformation of $\Gamma$ inside $\mathrm{SO}(4,\mathbb{C})$. 

Also, recall that, in Section \ref{ss3}, we studied the geometry of $\Gamma_{u}$ on the quadric $\mathrm{Q}_{2}$ and on $\Theta=\mathrm{Q}_{3}-\mathrm{Q}_{2}$ (see Propositions \ref{p1} and \ref{l0}). In this Section, we use these results to study the geometry of $\Gamma_{u}$ on $\mathrm{Q}_{3}$.  \\

Now, we will use the last  Proposition and the Cartan decomposition  for $\mathrm{SO}(4,\mathbb{C})$ (see \cite[p. 397]{Knapp} and \cite[pp. 74,75]{MayraTD}) to compute all the accumulation points of orbits of compact sets of $\mathrm{Q}_{3}$ for discrete subgroups of $\mathrm{SO}(4,\mathbb{C})$.

\begin{prop}[The Cartan decomposition] \label{cardec}
Let us denote by $A^{+}$  the group of all matrices of the form
\begin{equation} \label{am}
\left( \begin{array}{ccccc}
           e^{\lambda}   &    0        & 0 &       0     &     0         \\
            0            &   e^{\mu}   & 0 &       0     &     0          \\
            0            &    0        & 1 &       0     &     0           \\
            0            &    0        & 0 &   e^{-\mu}  &     0           \\
            0            &    0        & 0 &       0     &  e^{-\lambda}
            \end{array}  \right),
\end{equation} where $\lambda,\mu\in\mathbb{R}$.
Then, there exists a maximal compact subgroup $K$ of 
 $\mathrm{SO}(4,\mathbb{C})$, such that,
$$\mathrm{SO}(4,\mathbb{C})=\mathrm{K} A^{+} \mathrm{K}.
$$
\end{prop}
Let us consider the Cartan decomposition of $\mathrm{SO}(4,\mathbb{C})$. \\

Let us consider the compactifications  $\mathbb{R} \cup \{ - \infty \} \cup \{ + \infty \}$ and $\overline{\mathbb{R}}:=\mathbb{R} \cup \{ \infty \}$ of $\mathbb{R}$.
Let  $(g_{n})$ be a sequence of   $\mathrm{SO}(4,\mathbb{C})$ and $g_{n}=u_{n} \widetilde{a}_{n} \widetilde{u}_{n}$, be the Cartan decomposition of $g_{n}$, where  $u_{n},
\widetilde{u}_{n} \in K, \widetilde{a}_{n} \in A^{+}$.

We claim that there exists $i \in \mathrm{O}(4, \mathbb{C})$, such that,
 \begin{equation} \label{am}
a_{n}:= i^{-1} \widetilde{a}_{n} i =\left( \begin{array}{ccccc}
           e^{\lambda_{n}}   &       0         & 0 &       0         &     0         \\
            0                &   e^{\mu_{n}}   & 0 &       0         &     0          \\
            0                &       0         & 1 &       0         &     0           \\
            0                &       0         & 0 &   e^{-\mu_{n}}  &     0           \\
            0                &       0         & 0 &       0         &  e^{-\lambda_{n}}
            \end{array}  \right),
\end{equation}
where, if  $(\lambda_{n})$ or $(\mu_{n})$ converges to $\infty$, then it converges to $+ \infty$. Also,  $i$ is the identity matrix, or it is an element of $\mathrm{O}(4,\mathbb{C}) - \mathrm{SO}(4,\mathbb{C})$, such that, if restricted to $\mathrm{Q}_{2}$, can be represented in coordinates as
\begin{equation} \label{propi}
(x,y) \mapsto  \big( g(y),f(x) \big),
\end{equation}
where $f,g \in\mathrm{SL}(2,\mathbb{C})$. In particular, in this last case, $i$ interchanges the direction of the light geodesics contained in $\mathrm{Q}_{2}$; the reader can consult \cite[pp. 76, 77]{MayraTD} for an example of this.  

We say that $(g_{n})$ \emph{tends simply to infinity} if:
\begin{itemize}
\item The sequences $(u_{n})$ and $(\widetilde{u}_{n})$ converge,
\item The sequences $(\lambda_{n})$, $(\mu_{n})$ and $(\lambda_{n}-\mu_{n})$, converge in
$\overline{\mathbb{R}}$.
\end{itemize}

Following Frances in \cite[p. 8]{Frances}, if $(g_{n})$ is a sequence of $\mathrm{O}(4,\mathbb{C})$ which tends simply to infinity, we say that $(g_{n})$ is of
\emph{balanced distortion} if $(\lambda_{n})$ and $(\mu_{n})$ converge to $+\infty$  and $(\lambda_{n}-\mu_{n})$ converges to a point in $\mathbb{R}$.

We say that $(g_{n})$ is of  \emph{bounded distortion} if one of the sequences $(\lambda_{n})$ and $( \mu_{n} )$ converge to $+\infty$ and the other converges to a point in $\mathbb{R}$. We say that $(g_{n})$ is of
 \emph{mixed distortion} if $(\lambda_{n})$ and $(\mu_{n})$ converge to $+\infty$ and $(\lambda_{n}-\mu_{n})$ converge to $\infty$.

It is clear that every sequence of $\mathrm{SO}(4,\mathbb{C})$ which tends simply to infinity is of one of these kinds.

Now, we will show that this classification corresponds to the different dynamics of the orbits of compact sets in $\mathrm{Q}_{3}$.

Also, from the definitions, it follows that, in order to compute the accumulation points of the orbits of the compact sets in $\mathrm{Q}_{3}$ for discrete subgroups of $\mathrm{O}(4,\mathbb{C})$, it is enough to consider only the sequences which tend simply to infinity.

\begin{prop} \label{sdb}
If $(g_{n})$ is a sequence of balanced distortion of $\mathrm{SO}(4,\mathbb{C})$. Then, there exist two parallel light geodesics
$\Delta^{+}$ and $\Delta^{-}$ of $\mathrm{Q}_{2}$, such that
\begin{enumerate} \label{sdba}
\item $D_{(g_{n})}(Q_{3}-\Delta^{-})=\Delta^{+}$
\item For all $q \in \Delta^{+}$, there exists a light geodesic $l_{q} \subset \mathrm{Q}_{2}$, which is transversal to $\Delta^{-}$ such that $$l_{q}-\Delta^{-} \subset D_{(g_{n})}(q).$$
If $p \neq q$, then $l_{q} \cap l_{p} =\emptyset$. Also, the collection $\{ l_{q} \}_{q \in \Delta^{+}}$ foliates $\mathrm{Q}_{2}-\Delta^{-}$.
\item If the sequence $(g_{n})$ is of the form (\ref{am}); then, the function that assigns to each  $q \in \Delta^{+}$ its corresponding $l_{q}$, constructed in part 2 of this Proposition, is a  M\"{o}bius transformation.
\end{enumerate}
\end{prop}

\begin{proof} We will first consider sequences of the form (\ref{am}); in other words, let
\begin{displaymath}
a_{n}:= \left( \begin{array}{ccccc}
            e^{\lambda_{n}}  & 0  & 0    & 0 & 0 \\
            0  &  e^{\mu_{n}}  & 0 & 0 & 0  \\
            0 & 0 & 1 & 0 & 0 \\
            0 & 0 & 0 &  e^{-\mu_{n}} & 0  \\
           0 & 0 & 0 & 0 & e^{-\lambda_{n}}
           \end{array}
    \right),
\end{displaymath}
where $(\lambda_{n})$ and $(\mu_{n})$ are sequences of $\mathbb{R}$ which converge to $+\infty$ and $(\lambda_{n}-\mu_{n})$ converges in $\mathbb{R}$. In order to prove part 1 of this Proposition, we will prove that, for all $y \in \mathrm{Q}_{3}-\Delta^{-}$, there exists a point $p \in\Delta^{+}$, such that,  $D_{(g_{n})}(y)=p$ and, for all $p \in\Delta^{+}$,  there exists a point $y \in \mathrm{Q}_{3}-\Delta^{-}$, such that, $D_{(g_{n})}(y)=p$.

Let  $\nabla^{+}$ and $\nabla^{-}$ be the light geodesics  $\pi \big( \langle e_{1},e_{2} \rangle \big)$ and $\pi \big( \langle e_{4},e_{5} \rangle \big)$, respectively, and
$\delta:= \underset{n \to \infty}{\lim} (\lambda_{n}-\mu_{n})$.

Recall the definition of $C_{4}$ given in (\ref{defc}).
The submersion
$$\bar{s}: C^{4} - \langle e_{4},e_{5} \rangle  \to \mathbb{C}^{2}, \  \bar{s}(x_{1},x_{2},x_{3},x_{4},x_{5}):=(x_{1},e^{-\delta}x_{2}),$$ is well-defined in the quotient  $s: \mathrm{Q}_{3} -\nabla^{-} \to \nabla^{+}, \ s \big( [x] \big):= \big[\bar{s}(x) \big]$.

Consider $y=[z_{1}:z_{2}:z_{3}:z_{4}:z_{5}] \notin \nabla^{-}$; then, $z_{1} \neq 0$ or $z_{2} \neq 0$. Take  $z_{1} \neq 0 $ (the other case is similar), then, we can suppose that $z_{1}=1$. Let $(y_{n})$ be any sequence in $\mathrm{Q}_{3}$ which converges to $y$, then, for all $n \in\mathbb{N}$,
$$y_{n}=[1:y_{n}^{(2)}:y_{n}^{(3)}:y_{n}^{(4)}:y_{n}^{(5)}],
$$
where for all $j=2, \dots,5$, $\underset{n \to\infty}{\lim} y_{n}^{(j)}=z_{j}$. Then,
$$a_{n}y_{n}=[e^{\lambda_{n}}:e^{\mu_{n}}y_{n}^{(2)}:y_{n}^{(3)}:
e^{-\mu_{n}} y_{n}^{(4)}: e^{-\lambda_{n}}y_{n}^{(5)}]= [1:e^{\mu_{n}-\lambda_{n}}y_{n}^{(2)}:e^{-\lambda_{n}}y_{n}^{(3)}:
e^{-\mu_{n}-\lambda_{n}} y_{n}^{(4)}: e^{-2\lambda_{n}}y_{n}^{(5)}],
$$
and
$$\underset{n \to\infty}{\lim}a_{n}y_{n}=[1:e^{-\delta} z_{2}:0:0:0].
$$

Therefore, if $[z_{1}:z_{2}:z_{3}:z_{4}:z_{5}] \notin \nabla^{-}$, then,
\begin{equation} \label{auxp11}
D_{(a_{n})} \big([z_{1}:z_{2}:z_{3}:z_{4}:z_{5}] \big)=[z_{1}:e^{-\delta} z_{2}:0:0:0]=s \big( [z_{1}:z_{2}:z_{3}:z_{4}:z_{5}]  \big).
\end{equation}

Also, any point $p \in\nabla^{+}$ is equal to $[z_{1}:e^{-\delta} z_{2}:0:0:0]$, for some $z_{1},z_{2} \in \mathbb{C}$. This proves part 1 of the Proposition for sequences of the form (\ref{am}).

Recall the $\mathrm{SO}(4,\mathbb{C})$-equivariant biholomorphism (\ref{de2}). Let $z_{2} \in\mathbb{C}$, $p=[1: z_{2}:0:0:0]=\big( [1:0], [z_{2}:1] \big) \in \nabla^{+}$;  if
$$
l_{p}:=\Big\{ [1:z_{2}e^{\delta}:0:y_{4}:z_{2}y_{4}e^{\delta} ] : y_{4} \in\mathbb{C} \Big\} = \Big\{ \big( [1:y_{4}], [-z_{2}e^{\delta}:1] \big): y_{4} \in\mathbb{C} \Big\},
$$
then,  by (\ref{auxp11}), $l_{p} \subset D_{(g_{n})}(p)$. This proves part 2 of this Proposition for sequences of the form (\ref{am}).

Also, by the latter, it is clear that, there exist parametrizations of  $\nabla^{+}$ and of the space of horizontal light geodesics, such that, the function that assigns to each $q \in \Delta^{+}$ its corresponding $l_{q}$ is of the form:
\begin{eqnarray*}
\mathbb{CP}^{1} & \to & \mathbb{CP}^{1} \\
z & \to & -z e^{\delta}.
\end{eqnarray*}

This proves part $3$ of this Proposition. \\

For all  $n \in\mathbb{N}$, let $g_{n}=\widetilde{u}_{n} i  a_{n} i^{-1} \overline{u}_{n}$ be the Cartan decomposition of  $g_{n}$, where
\begin{itemize}
 \item $(a_{n})$ is of the form   (\ref{am}), 
 \item $(\lambda_{n})$ and $(\mu_{n})$ are  sequences that converge to $+\infty$, 
 \item $(\lambda_{n}-\mu_{n})$ is a sequence that converges to a point in $\mathbb{R}$,
 \item  $(\widetilde{u}_{n})$ and $(\overline{u}_{n})$ are convergent sequences of  $\mathrm{SO}(4,\mathbb{C})$ 
 \item $i$ is the identity matrix or it is an element of  $\mathrm{O}(4,\mathbb{C})-\mathrm{SO}(4,\mathbb{C})$ of the form  (\ref{propi}). 
 \end{itemize}
 
 Let us define
$\widetilde{u}:=\lim \widetilde{u}_{n}, \ \overline{u}:=\lim \overline{u}_{n}$,  $\Delta^{+}:= \widetilde{u} \big( i ( \nabla^{+} \big) \big)$ and  $\Delta^{-}:= \overline{u}~^{-1}  \big(i( \nabla^{-} ) \big)$.  As $\mathrm{Q}_{2}$ is  $\mathrm{O}(4,\mathbb{C})$-invariant, then $\Delta^{+}$ and $\Delta^{-}$ are light geodesics  contained in   $\mathrm{Q}_{2}$. As $\mathrm{SO}(4,\mathbb{C})$ fixes the direction of the light geodesics contained in $\mathrm{Q}_{2}$, by the properties of $i$, we have that  $\Delta^{+}$ and $\Delta^{-}$ are parallel light geodesics.

If $y \in\Delta^{-}$, then $i^{-1} \overline{u}(y) \in \nabla^{-}$. So, by part 1 of this Proposition for sequences of the form (\ref{am}), $D_{(a_{n})} \big( i^{-1} \overline{u}(y) \big) \in \nabla^{+}$. Then, by Proposition \ref{paokak}, $$D_{(\widetilde{u_{n}}ia_{n}i^{-1}\overline{u}_{n})}(y)=\widetilde{u}~i \Big( D_{(a_{n})} (i^{-1} \overline{u}(y) \big) \Big) \in\Delta^{+}.$$ 
If $q \in \Delta^{+}$, then $i^{-1} \widetilde{u}^{-1}(q) \in \nabla^{+}$. So, by part 1 of this Proposition for for sequences of the form (\ref{am}), there exists $y \in\mathrm{Q}_{3}-\nabla^{-}$, such that, $D_{(a_{n})}(y)=i^{-1} \widetilde{u} ~ ^{-1}(q)$. Then, by Proposition \ref{paokak}, 
$$D_{ ( \widetilde{u}_{n}i a_{n} i^{-1} \overline{u}_{n})} \big( \overline{u}~ ^{-1} i(y) \big)= \widetilde{u} ~ i \big( D_{(a_{n})}(y) \big)= \widetilde{u} ~ i \big( i^{-1} \widetilde{u}~ ^{-1}(q) \big)=q.
$$

This proves part $1$ of this Proposition for all sequences. \\

Let $p \in\Delta^{+}$, define  $q :=i^{-1} \big( \widetilde{u}~^{-1} (p) \big) \in \nabla^{+}$. By part $2$  of this Proposition for sequences of the form (\ref{am}),  there exists a light geodesic  $l_{q} \subset \mathrm{Q}_{3}$ that is transversal to
 $\nabla^{-}$, such that, $l_{q}-\nabla^{-} \subset D_{(a_{n})}(q)$. Also, the collection of all light geodesics
 $l_{q}$, constructed this way, is a folliation of $\mathrm{Q}_{3}-\nabla^{-}$. By Proposition  \ref{paokak} and, as   $\mathrm{O}(4,\mathbb{C})$ sends light geodesics to light geodesics, $l_{p}:=\overline{u}~^{-1} \big( i ( l_{q}) \big)$ is a light geodesic that is transversal to $\Delta^{-}$, such that, $l_{p}-\Delta^{-} \subset D_{(g_{n})}(p)$ and the collection of all light geodesics $l_{q}$, constructed this way, is a folliation of  $\mathrm{Q}_{2}-\Delta^{-}$. This proves part $2$ of the Proposition for all sequences.
 \end{proof}

Recall that if $p \in\mathrm{Q}_{3}$, in Section \ref{s2} we defined  $C(p)$ as the union of all light geodesics in $\mathrm{Q}_{3}$ which contain $p$, $\widetilde{C}(p)$ as the space of all the light geodesics which contain $p$ and we showed that $\widetilde{C}(p)$ is a complex manifold biholomorphic to the Riemann sphere. Let us also recall that, by definition (see Subection \ref{aut}), $\Theta$ is a manifold $\big( \mathrm{SL}(2,\mathbb{C}) \times \mathrm{SL}(2,\mathbb{C}) \big)$-equivariantly biholomorphic to $\mathrm{SL}(2,\mathbb{C})$.

Now, we will study the dynamics of the second kind of sequence which diverges simply to infinity. As we will see later in this Section, the group $\Gamma_{u}$, defined in (\ref{gamma}), for $u$ sufficiently close to the constant map (see Section \ref{ss3}), does not admit sequences of this kind.

\begin{prop} \label{sda}
 Let $( g_{n} )$ be a sequence of bounded distortion of $\mathrm{SO}(4,\mathbb{C})$. Then, there exist two points $p_{+}, p_{-}$ in  $\mathrm{Q}_{2}$ and a biholomorphism
 $\bar{g}_{\infty}$ from the space $\widetilde{C}(p_{-})$ onto the space $\widetilde{C}(p_{+})$, such that
\begin{enumerate}
 \item If $q \notin C^{-}$, then $D_{(g_{n})} ( q )=\{p^{+}\}$.
 \item If $q \in C^{-} - \{p_{-} \}$, then, $D_{(g_{n})} ( q )$ is the image, under  $\bar{g}_{\infty}$, of the light geodesic which contains  $p_{-}$ and $q$ .
  \item There exist points in  $\Theta$ which are dynamically related to each other.
\end{enumerate}
\end{prop}
\begin{proof} We will first consider sequences of the form (\ref{am}): let us consider
\begin{displaymath}
a_{n}= \left( \begin{array}{ccccc}
            e^{\lambda_{n}}  & 0  & 0    & 0 & 0 \\
            0  &  e^{\mu_{n}}  & 0 & 0 & 0  \\
            0 & 0 & 1 & 0 & 0 \\
            0 & 0 & 0 &  e^{-\mu_{n}} & 0  \\
           0 & 0 & 0 & 0 & e^{-\lambda_{n}}
           \end{array}
    \right),
\end{displaymath}
where $\lambda_{n} \to +\infty$ and $\mu_{n} \to \mu_{\infty} < \infty$. Let $q_{-}:=[0:0:0:0:1]$ and
$q_{+}:=[1:0:0:0:0]$.

The map
\begin{displaymath}
h_{\infty}:= \left( \begin{array}{ccccc}
            0                   & 0                   & 0    & 0                  & 1      \\
            0                   & e^{\mu_{\infty}}    & 0    & 0                  & 0      \\
            0                   & 0                   & 1    & 0                  & 0      \\
            0                   & 0                   & 0    & e^{-\mu_{\infty}}  & 0      \\
            1                   & 0                   & 0    & 0                  & 0
           \end{array}
    \right)
\end{displaymath}
belongs to $\mathrm{SO}(5,\mathbb{C})$ and, by Proposition \ref{lg}, induces a biholomorphism
\begin{displaymath}
\bar{h}_{\infty}:= \left( \begin{array}{ccc}
            e^{\mu_{\infty}}    & 0    & 0  \\
            0  & 1  & 0  \\
            0 & 0 & e^{-\mu_{\infty}}
           \end{array}
    \right)
\end{displaymath}
from $\widetilde{C}(q_{-})$ onto $\widetilde{C}(q_{+})$, which sends the light geodesic
$$\overline{\pi \big( \{0\} \times \{z_{2}\} \times \{z_{3} \} \times \{z_{4} \} \times \mathbb{C} \big)},$$
which contains $q_{-}$ and \mbox{$[0:z_{2}:z_{3}:z_{4}:0]$} to the light geodesic
$$\overline{\pi \big( \mathbb{C}  \times \{z_{2} e^{\mu_{\infty}} \} \times \{z_{3} \} \times \{z_{4}  e^{-\mu_{\infty}} \} \times \{0\} \big)},$$
which contains $q_{+}$ and $[0:z_{2} e^{\mu_{\infty}}:z_{3}:z_{4}e^{-\mu_{\infty}}:0]$.

Let $q=[z_{1}:z_{2}:z_{3}:z_{4}:z_{5}] \notin C(q_{-})$; then $z_{1} \neq 0$, so we can suppose that $z_{1}=1$.
 Let $(y_{n})$ be any sequence of $\mathrm{Q}_{3}$ which converges to $q$. Then, for $n \in\mathbb{N}$ sufficiently large, $$y_{n}=[1:y_{n}^{(2)}:y_{n}^{(3)}:y_{n}^{(4)}:y_{n}^{(5)}],$$
and, hence,
$$a_{n}y_{n}=[e^{\lambda_{n}}: y_{n}^{(2)} e^{\mu_{n}}: y_{n}^{(3)}:
y_{n}^{(4)} e^{-\mu_{n}} :
y_{n}^{(5)} e^{-\lambda_{n}}]=
[1: y_{n}^{(2)} e^{\mu_{n}-\lambda_{n}}: y_{n}^{(3)} e^{-\lambda_{n}}:
y_{n}^{(4)}  e^{-\mu_{n}-\lambda_{n}} :
y_{n}^{(5)} e^{-2\lambda_{n}}].
$$

Therefore, if $q \notin C(q_{-})$,
$$
\underset{n \to\infty}{\lim}a_{n}y_{n}=[1:0:0:0:0],
$$
and, hence,
\begin{equation}
D_{(a_{n})} \big( q \big) = \big\{ [1:0:0:0:0] \big\}.
\end{equation}

This proves part 1 of this Proposition for sequences of the form (\ref{am}).

Let $q=[0:z_{2}:z_{3}:z_{4}:z_{5}] \in C(q_{-})-\{q_{-}\}$; then $z_{2} \neq 0$ or $z_{4} \neq 0$. Take $z_{2} \neq 0$ (the other case is similar); then, we can suppose that $q=[0:1: z_{3}: z_{4}:z_{5}]$.

Consider any sequence $(y_{n})$ in $\mathrm{Q}_{3}$ which converges to $q$, then, for $n$ sufficiently large,
$$y_{n}=[y_{n}^{(1)}:1: y_{n}^{(3)}:y_{n}^{(4)}:y_{n}^{(5)}]$$
and
$$y_{n}^{(1)} \to 0, \ y_{n}^{(3)} \to z_{3}, \ y_{n}^{(4)} \to z_{4}, y_{n}^{(5)} \to z_{5}.$$

We have that
$$a_{n}y_{n}=[e^{\lambda_{n}}y_{n}^{(1)}:
e^{\mu_{n}}:
 y_{n}^{(3)}:
y_{n}^{(4)} e^{-\mu_{n}}:y_{n}^{(5)}e^{-\lambda_{n}}].
$$

Suppose that $(a_{n}y_{n})$ is convergent; then, $e^{\lambda_{n}}y_{n}^{(1)} \to b$, for some $b \in \mathbb{C}$, or $e^{\lambda_{n}}y_{n}^{(1)} \to \infty$.

If $e^{\lambda_{n}}y_{n}^{(1)} \to b$,
then,
$$a_{n}y_{n} \to [b: e^{\mu_{\infty}}: z_{3}:e^{-\mu_{\infty}}z_{4}:0].
$$

If $e^{\lambda_{n}}y_{n}^{(1)} \to \infty$, then,
$$a_{n}y_{n} \to [1:0:0:0:0].
$$

Therefore,
$$D_{(a_{n})}(q) \subset \pi \big( \mathbb{C} \times \{ z_{2} e^{\mu_{\infty}} \} \times \{z_{3}\} \times \{z_{4} e^{-\mu_{\infty}} \} \times \{ 0 \} \big).
$$

Let us consider any $b \in\mathbb{C}$ and the sequence $(y_{n})$ in $\mathrm{Q}_{3}$ defined by
$$
y_{n}:=[be^{-\lambda_{n}}:1:z_{3}:y_{n}^{(4)}:z_{5}],
$$
 where $y_{n}^{(4)}:=be^{-\lambda_{n}}z_{5}-z_{3}^{2}$.
 
Then,
$$a_{n}y_{n} \to \big[ b:e^{\mu_{\infty}}z_{2}:z_{3}:e^{-\mu_{\infty}}z_{4}:0 \big],$$
so we get
$$D_{(a_{n})}(q) = \pi \big( \mathbb{C} \times \{ z_{2} e^{\mu_{\infty}} \} \times \{z_{3}\} \times \{z_{4} e^{-\mu_{\infty}} \} \times \{ 0 \} \big).
$$

This proves part 2 of this Proposition for sequences of the form (\ref{am}).

Also, if  $q \in \Theta$, that is $z_{3} \neq 0$, then $D_{(a_{n})}(q)$ intersects to $\Theta$; then, we get part $3$ also for sequences  of the form (\ref{am}). \\

For all   $n \in\mathbb{N}$, let  $g_{n}=\widetilde{u}_{n} i  a_{n} i^{-1} \overline{u}_{n}$ be the Cartan decomposition of  $g_{n}$, where
\begin{itemize}
\item $a_{n}$ is of the form  (\ref{am}),
\item $(\lambda_{n})$ is a sequence which converges to $+\infty$,
\item $(\mu_{n})$ is a sequence that converges in $\mathbb{R}$,
\item $(\widetilde{u}_{n})$ and $(\overline{u}_{n})$ are convergent sequences in  $\mathrm{SO}(4,\mathbb{C})$
\item $i$ is either the identity matrix, or an element of $\mathrm{O}(4,\mathbb{C})-\mathrm{SO}(4,\mathbb{C})$ of the form (\ref{propi}).
\end{itemize}

 Let us define  $p_{-}:=\overline{u}^{-1}(q_{-})$, $p_{+}:=\widetilde{u}(q_{+})$, $\overline{u}:=\lim \overline{u}_{n}$, $\widetilde{u}:=\lim \widetilde{u}_{n}$
 and $g_{\infty}:= \widetilde{u}  \ i  \ h_{\infty} i^{-1} \  \overline{u}$. Observe that $g_{\infty}$ is an element of  $\mathrm{O}(4,\mathbb{C})$ and induces a biholomorphism  $\bar{g}_{\infty}$ from $\widetilde{C}(p_{-})$ onto $\widetilde{C}(p_{+})$.

Then, since  $\mathrm{O}(4,\mathbb{C})$ sends light geodesics to light geodesics and light cones to light cones,  Proposition  \ref{paokak} implies part $2$ in the general case. Also, as  $\widetilde{u} i, i^{-1} \overline{u} \in \mathrm{O}(4,\mathbb{C})$, $C(p_{+})$ and $C(p_{-})$ are not contained in $\mathrm{Q}_{2}$ and $\Theta$ is $\mathrm{O}(4,\mathbb{C})$-invariant, by part $3$ for sequences of the form  (\ref{am}), we get part $3$ for all sequences.
\end{proof}

\begin{prop} \label{sdm}
Suppose $(g_{n})$ is a sequence of mixed distortion of $\mathrm{SO}(4,\mathbb{C})$.  Then, there exist  in $\mathrm{Q}_{2}$ two points $p_{+}$ and $p_{-}$ and two parallel light geodesics  $\Delta^{+}$ and $\Delta^{-}$ which contain $p_{+}$ and $p_{-}$, respectively,  such that
 \begin{enumerate}
\item If $q \notin C(p_{-})$, then $D_{(g_{n})} ( q )= \{p_{+}\}$.
\item If $q \in C(p_{-})
-  \Delta^{-} $, then $D_{(g_{n})} ( q )=\Delta^{+}$.
\item If $q \in \Delta^{-}-\{p_{-}\}$, then $D_{(g_{n})}(q)=C(p_{+})$.
\end{enumerate}
\end{prop}
\begin{proof} Part 1 of this Proposition follows as in Proposition \ref{sda}. 

In order to prove part 2 of this Proposition, we will first consider sequences of the form (\ref{am}), i.e., let us consider
\begin{displaymath}
a_{n}:= \left( \begin{array}{ccccc}
            e^{\lambda_{n}}    & 0    & 0 & 0 & 0 \\
            0  &  e^{\mu_{n}}  & 0 & 0 & 0 \\
            0 & 0 & 1 & 0 & 0 \\
            0 & 0  & 0 &  e^{-\mu_{n}} & 0  \\
           0 & 0 & 0 & 0 & e^{-\lambda_{n}}
           \end{array}
    \right),
\end{displaymath}
where $(\mu_{n}), (\lambda_{n})$ and $(\lambda_{n}-\mu_{n})$ converge to $\infty$. Let us suposse that $(\lambda_{n}-\mu_{n})$ converges to $+\infty$; the other case is similar.

Let $q_{-}:=[0:0:0:0:1], q_{+}:=[1:0:0:0:0]$ and $\nabla^{+}$ and $\nabla^{-}$ be the light geodesics $\pi \big( \langle e_{1},e_{2} \rangle \big)$ and $\pi \big( \langle e_{4},e_{5} \rangle \big)$, respectively, contained in $\mathrm{Q}_{2}$.

If $q \in C(q_{-}) -\nabla^{-}$,  assume that $q=[0:1:z_{3}:z_{4}:z_{5}] \in \mathrm{Q}_{3}$. Let us consider any sequence $(y_{n})$ in $\mathrm{Q}_{3}$ which converges to $q$; then, for  $n  \in\mathbb{N}$ sufficiently large,
$y_{n}=[y_{n}^{(1)}:1:y_{n}^{(3)}:y_{n}^{(4)}:y_{n}^{(5)} ]$,
where
$$y_{n}^{(1)} \to 0, \ y_{n}^{(3)} \to z_{3}, \ y_{n}^{(4)}\to z_{4}, \ y_{n}^{(5)} \to z_{5}.
$$
So, if $q \in C(q_{-}) -\nabla^{-}$, we have
$$ a_{n}y_{n}=[e^{\lambda_{n}} y_{n}^{(1)}:e^{\mu_{n}}: y_{n}^{(3)}: y_{n}^{(4)} e^{-\mu_{n}} :
y_{n}^{(5)} e^{-\lambda_{n}}] =  [e^{\lambda_{n}-\mu_{n}} y_{n}^{(1)}:1: y_{n}^{(3)} e^{-\mu_{n}}: y_{n}^{(4)} e^{-2\mu_{n}} :
y_{n}^{(5)} e^{-\lambda_{n}-\mu_{n}}].
$$

If $(a_{n}y_{n})$ converges, then $(e^{\lambda_{n}-\mu_{n}} y_{n}^{(1)})$ converges in $\mathbb{R} \cup \{\infty\}$.

If $e^{\lambda_{n}-\mu_{n}} y_{n}^{(1)} \to \infty$, then $a_{n}y_{n} \to [1:0:0:0:0]$.

If $e^{\lambda_{n}-\mu_{n}} y_{n}^{(1)} \to a \in \mathbb{R}$, then $a_{n}y_{n} \to [a:1:0:0:0]$.

Therefore, if $q \in C(q_{-}) -\nabla^{-}$
$$D_{(a_{n})} \big( q \big) \subset \nabla^{+}.
$$
Let us consider any $a \in \mathbb{R}$ and the sequence
$$y_{n}=[ a e^{-\lambda_{n}+\mu_{n}}:1:z_{3}:y_{n}^{(4)}:z_{5}], \qquad y_{n}^{(4)}:= a e^{-\lambda_{n}+\mu_{n}}z_{5}-z_{3}^{2},
$$
in $\mathrm{Q}_{3}$.
Then, for $n \in\mathbb{N}$ sufficiently large,
$$ a_{n}y_{n}=[a e^{\mu_{n}}: e^{\mu_{n}}: z_{3} : y_{n}^{(4)} e^{-\mu_{n}}:
e^{-\lambda_{n}} z_{5}
] =  [a: 1: e^{-\mu_{n}} z_{3}: e^{-2\mu_{n}} y_{n}^{(4)}: e^{-\mu_{n}-\lambda_{n}} z_{5}].$$

Then, $a_{n}y_{n} \to [a:1:0:0:0]$ and
$$D_{(a_{n})} \big( q \big)=\nabla^{+}.$$

This proves part 2 of this Proposition for sequences of the form (\ref{am}).

 For all  $n \in\mathbb{N}$, let $g_{n}=\widetilde{u}_{n} i a_{n} i^{-1} \overline{u}_{n}$ be the Cartan decomposition of $g_{n}$, where  
 \begin{itemize}
 \item  $a_{n}$ is of the form  (\ref{am}), 
 \item $(\lambda_{n})$, $(\mu_{n})$ and $(\lambda_{n}-\mu_{n})$ are sequences that converge to $+\infty$,
 \item $(\widetilde{u}_{n})$ and $(\overline{u}_{n})$ are convergent sequences of  $\mathrm{SO}(4,\mathbb{C})$,
 \item $i$ is the identity matrix, or it is an element of  $\mathrm{O}(4,\mathbb{C})-\mathrm{SO}(4,\mathbb{C})$ of the form  (\ref{propi}).
 \end{itemize}
 
  Let us define  $p_{+}:=\widetilde{u} ~ i (q_{+}), p_{-}:=\overline{u}^{-1} ~ i (q_{-})$  and the light geodesics $\Delta^{+}:=\widetilde{u} \big( i(\nabla^{+}) \big)$ and $\Delta^{-}:= \overline{u}~^{-1} \big(i(\nabla^{-}) \big)$. As $\mathrm{Q}_{2}$ is $\mathrm{O}(4,\mathbb{C})$-invariant, then $\Delta^{+}$ and $\Delta^{-}$  are light geodesics  contained in  $\mathrm{Q}_{2}$.  As $\mathrm{SO}(4,\mathbb{C})$ fix the direction of the light geodesics contained in $\mathrm{Q}_{2}$, if $i$ is not the identity, it reverses the direction of them; then, we have that  $\Delta^{+}$ and $\Delta^{-}$ are both horizontal or vertical. By Proposition \ref{paokak}, part 2 of this Proposition for sequences of the form (\ref{am}) and as  $\mathrm{O}(4,\mathbb{C})$ sends light geodesics to light geodesics and light cones to light cones, we get part $2$ of this Proposition in the general case. 
  
  The proof of part $3$ is similar to the proof of part $2$.
\end{proof}

Recall the group $\Gamma_{u}$,  which was defined  in  (\ref{gamma}) and studied in Section \ref{ss3}.
The next Proposition is a complex analog of the  limit set defined by Frances (see \cite[p. 894]{Frances}).
\begin{prop} \label{fls}
 Let \mbox{$\Gamma \subset \mathrm{SL}(2,\mathbb{C})$} be  a torsion free,  (classical) Kleinian group. If $u:\Gamma \to \mathrm{SL}(2,\mathbb{C})$ is a group morphism such that $\Gamma_{u}$ acts properly discontinuously on $\Theta$, then $\Gamma_{u}$ has no sequences of bounded distortion and  acts properly discontinuously on the complement in $\mathrm{Q}_{3}$ of
$$
\Lambda_{F}:= \overline{ \underset{(g_{n})}{\bigcup} \big( \Delta^{+}(g_{n}) \cup \Delta^{-}(g_{n}) \big)} \subset \mathrm{Q}_{2},
$$
where the last union is taken over all the sequences
 $(g_{n})$ of balanced or mixed distortion of $\Gamma_{u}$ and
$\Delta^{+}(g_{n})$ and $\Delta^{-}(g_{n})$ denote the limit light geodesics which correspond to the sequence
$(g_{n})$.
\end{prop}
\noindent\textbf{Remark.} By Proposition \ref{p1}, if \mbox{$\Gamma \subset \mathrm{SL}(2,\mathbb{C})$} is  a torsion free, convex-cocompact,  (classical) Kleinian group and $u:\Gamma \to \mathrm{SL}(2,\mathbb{C})$ is a group morphism sufficiently close to the constant morphism, then $\Gamma_{u}$ acts properly discontinuously on $\Theta$, and so, the hypothesis of the Proposition are valid.

\begin{proof}
By Proposition \ref{sda}, every sequence of bounded distortion has points dynamically related in $\Theta$, it follows that $\Gamma_{u}$ does not contain such sequences. It follows that every sequence of $\Gamma_{u}$ which diverges simply to infinity is of balanced or mixed distortion.

By Propositions~\ref{sdb} and \ref{sdm}, every sequence of balanced or mixed distortion has two limit light geodesics associated to it (contained in $\mathrm{Q}_{2}$), one attractor and the other repeller, such that  the accumulation points of the orbits (associated with this sequence and with the sequence formed by the inverses)
of the compact sets in $\mathrm{Q}_{3}$, which do not intersect these two light geodesics, lie in these light geodesics. Then, $\Gamma_{u}$ acts properly discontinuously on the complement in $\mathrm{Q}_{3}$ of $\Lambda_{F}$.
\end{proof}

\textbf{Proof of Theorem \ref{res1}}
Let $\Gamma$ be a  torsion free, finitely-generated, (classical) Kleinian group  with domain of discontinuity $\Omega$ in $\mathbb{CP}^{1}$ and $u:\Gamma \to \mathrm{SL}(2,\mathbb{C})$  a group morphism.
Recall the definition of
$$\Gamma_{u}:= \Big\{ \big( \gamma,u(\gamma) \big): \gamma\in\Gamma \Big\},
$$
given in (\ref{gamma}).  Recall also that $\Gamma_{u}$ is torsion free.

Suppose that $\Gamma_{u}$ acts properly discontinuously on $\Theta$.   By Proposition \ref{fls},  all the sequences of $\Gamma_{u}$ which tend simply to infinity, are of balanced or mixed distortion and $\Gamma_{u}$ acts properly discontinuously on $\mathrm{Q}_{3}-\Lambda_{F}$, where $\Lambda_{F} \subset \mathrm{Q}_{2}$ was defined in the same Proposition.

We will prove that $U_{\Gamma} \subset \mathrm{Q}_{3}-\Lambda_{F}$, or equivalently, that
\begin{equation} \label{ast}
\big( \Omega \times\mathbb{CP}^{1} \big) \subset \big( \mathrm{Q}_{2} - \Lambda^{F} \big).
\end{equation}

First, we will prove that the light geodesics of  $\Lambda_{F}$ are vertical, as in the case of A. Guillot in \cite[pp. 224, 225]{GuillotD}, where  $u$ is the constant morphism. Supose that  $\Delta^{+}$ and $\Delta^{-}$ are the limit light geodesics which correspond to the divergent sequence $\big( g_{n}, u(g_{n}) \big)$ and that $\Delta^{+}$ and $\Delta^{-}$ are horizontal. We have two cases:
\begin{enumerate}
\item Suppose that  $\big( g_{n}, u(g_{n}) \big)$ is of mixed distortion and $C^{+}$ and $C^{-}$ are their attractor and repellor limit light cones. Then, by Proposition  \ref{sdm}, any point  $x$ of $\big( \Omega \times \mathbb{CP}^{1} \big) \cap  \Delta^{-}$, different from $p_{-}$, is dynamically related to any point in  $C^{+} \cap  \big( \Omega \times \mathbb{CP}^{1} \big)$. In particular, as   $\Delta^{+} \subset C^{+}$, then  $x$ is dynamically related to any point in $\Delta^{+} \cap  \big( \Omega \times \mathbb{CP}^{1} \big)$.
\item Suppose that  $\big( g_{n}, u(g_{n}) \big)$  is of balanced distorsion. Then, by the Proposition   \ref{sdb}, for each point  $q$ in $ \Delta^{+}\cap  \big( \Omega \times \mathbb{CP}^{1} \big)$, there exists a vertical light geodesic  $l_{q} \subset\mathrm{Q}_{2}$, such that $q$ is dynamically related to any point in $l_{q}-\Delta^{-}$.
 We will prove that there exists  $q \in \Delta^{+} \cap  \big( \Omega \times \mathbb{CP}^{1} \big)$, such that, $l_{q} \subset  \Omega\times\mathbb{CP}^{1}$.

For all  $n \in\mathbb{N}$,  let $\widetilde{u}_{n} i a_{n} i^{-1} \overline{u}_{n}$ be the Cartan decomposition of $\big( g_{n}, u(g_{n}) \big)$, where 
\begin{itemize}
\item $a_{n}$ is of the form   (\ref{am}), 
\item $\lambda_{n} \to + \infty$, 
\item $\mu_{n} \to + \infty$,
\item  $(\lambda_{n}-\mu_{n})$ converge to a point in $\mathbb{R}$,  
\item $(\widetilde{u}_{n})$ and $(\overline{u}_{n})$ are convergent sequences in
$\mathrm{SO}(4, \mathbb{C})$
\item  $i$ is an element of  $\mathrm{O}(4,\mathbb{C})$ which, restricted to   $\mathrm{Q}_{2}$, can be represented in coordinates as  (\ref{propi}).
\end{itemize}

Now, we will translate what we want to prove for $(g_{n})$ to the same problem, but for the sequence  $(a_{n})$. If   $q \in\Delta^{+}$, recall the construction of   $l_{q}$ (see Proposition \ref{sdb}). Let  $\widetilde{u}:=\lim \widetilde{u}_{n}$ and $\overline{u}:=\lim \overline{u}_{n}$. Then, there exist    $y_{0},\bar{y}_{0} \in\mathbb{CP}^{1}$ such that $$\Delta^{+}=\big\{(x,y_{0}):x \in\mathbb{CP}^{1} \big\}$$ and $$\Delta^{-}=\big\{(x,\bar{y}_{0}):x \in\mathbb{CP}^{1} \big\}.$$
By the aforementioned properties of $i$ and, since the restriction of $\widetilde{u}^{-1}$ and $\overline{u}$ to $\mathrm{Q}_{2}$ can be represented as  (\ref{defacxc}),  we get that the function $i^{-1} \widetilde{u}^{-1}$ can be represented in coordinates as:
$$(x,y) \mapsto \big(v_{1}(y), v_{2}(x) \big),
$$
where $v_{1},v_{2} \in \mathrm{SL}(2,\mathbb{C})$ and the function  $i^{-1} \overline{u}$ can be represented in coordinates as:
$$(x,y) \mapsto \big( w_{1}(y), w_{2}(x) \big),
$$
where  $w_{1},w_{2} \in \mathrm{SL}(2,\mathbb{C})$.

Then, by the formula given by Proposition  \ref{paokak}, it is enough to prove that there exists a point
  $q$ in  $\nabla^{+} \cap \big( \mathbb{CP}^{1} \times v_{2}(\Omega) \big)$, such that the  $l_{q}$ that corresponds to the sequence $(a_{n})$, is contained in $\mathbb{CP}^{1} \times w_{2}(\Omega)$. Then, by part $3$ of the Proposition \ref{sdb}, this is equivalent to show that there exists a point in
  $$ g \ v_{2}(\Omega) \cap g \  w_{2}(\Omega),
  $$
where  $g$ is a M\"obius transformation.

But, by Ahlfors Thereom  (see \cite{Ahlfors}),  $\Lambda$ has Lebesgue measure zero and as the  M\"obius transformations preserve the sets of Lebesgue measure zero, this is true; it follows that
$$g \ v_{2}(\Lambda) \cup g \  w_{2}(\Lambda)
$$
has Lebesgue measure zero.
\end{enumerate}

 In any case, there exist two points in  $\Omega \times \mathbb{CP}^{1}$ that are dynamically related to each other and correspond to the sequence $\big( g_{n}, u(g_{n}) \big)$.   \\

On the other hand, by Proposition \ref{l0}, we know that  $\Gamma_{u}$ acts properly discontinuously and uniformly on  $\Omega\times\mathbb{CP}^{1}$; then, there are no points dynamically related to each other in $\Omega\times\mathbb{CP}^{1}$; this, however, contradicts the previous paragraph.
Therefore, all the limit light geodesics of   $\Gamma_{u}$  are vertical. \\

Assume now that (\ref{ast}) is not true, so there exists a point
$$(x,y) \in \big( \Omega\times\mathbb{CP}^{1} \big) \cap \Lambda_{F}.
$$

 As $\Omega$ is open in $\mathbb{CP}^{1}$, there exists an attractor limit light geodesic  $\Delta^{+}$ which corresponds to a sequence
$\big( g_{n}, u(g_{n}) \big)$ of balanced or mixed distortion of
 $\Gamma_{u}$, such that, $\Delta^{+} \cap \big( \Omega\times\mathbb{CP}^{1} \big)
\neq\emptyset$. As the limit light geodesics of $\Gamma_{u}$ are vertical, then $\Delta^{+} \subset \big( \Omega\times\mathbb{CP}^{1} \big)$.

We will show that there exist two points in
 $\Omega\times \mathbb{CP}^{1}$ which are dynamically related to each other, and this will be a contradiction because we know that the action is properly discontinuous on $\Omega\times\mathbb{CP}^{1}$. There are two cases:
\begin{enumerate}
\item If $\big( g_{n},u(g_{n}) \big)$ is a mixed distortion sequence, then by Proposition \ref{sdm}, there exist a repeller limit light geodesic $\Delta^{-}$ and two limit light cones $C^{-}$, $C^{+}$ of $\big( g_{n}, u(g_{n}) \big)$, such that, $\Delta^{-} \subset C^{-}$ and $\Delta^{+} \subset C^{+}$ and if $y \in C^{-} - \Delta^{-}$, then
  $D_{\big( g_{n}, u(g_{n}) \big)}(y)=\Delta^{+}$. By Proposition \ref{lg}, we know that
  $\big( C^{-} \cap \mathrm{Q}_{2} \big) -\Delta^{-}$ is a light geodesic of the form $\mathbb{CP}^{1} \times \{z\}$ (minus one point), for some $z \in\mathbb{CP}^{1}$. This light geodesic (minus one point) intersects  $\Omega\times\mathbb{CP}^{1}$, and any point of this intersection is dynamically related to any point of $\Delta^{+} \subset \big(\Omega\times\mathbb{CP}^{1} \big)$.
\item If $\big( g_{n},u(g_{n}) \big)$ is a balanced distortion sequence, then, by Proposition \ref{sdb}, any point of
     $\big(\Omega\times\mathbb{CP}^{1} \big)-\Delta^{-}$ is dynamically related to a point of $\Delta^{+} \subset \big(\Omega\times\mathbb{CP}^{1} \big)$.
\end{enumerate}
This proves (\ref{ast}).

Now, suppose that $\Gamma\setminus\Omega$ is compact. We will prove that $U_{\Gamma} \subset \mathrm{Q}_{3}-\Lambda_{F}$, or equivalently, that (\ref{ast}) is not only an inclusion, but an equality.

Recall the $\mathrm{SO}(4,\mathbb{C})$-equivariant biholomorphism (\ref{de2}) between $\mathrm{Q}_{2}$ and $\mathbb{CP}^{1}\times\mathbb{CP}^{1}$. By Proposition \ref{l0}, $\Gamma_{u} \setminus ( \Omega\times\mathbb{CP}^{1})$ is a manifold. 

The projection $(x,y) \mapsto x, x \in\Omega, y \in \mathbb{CP}^{1}$, is well defined in the quotient

$$\Gamma_{u}\setminus(\Omega \times\mathbb{CP}^{1}) \to \Gamma\setminus\Omega, \qquad \big[ (x,y) \big] \mapsto [x].$$

In fact, it is a locally trivial fibration with compact fiber, so it is a proper map. As $\Gamma\setminus \Omega$ is compact, then
$\Gamma_{u} \setminus (\Omega\times\mathbb{CP}^{1})$ is compact.

Suppose that there exists an invariant open set $U$, which contains $U_{\Gamma}$, where $\Gamma_{u}$ acts properly discontinuously. Then, $\Gamma_{u} \setminus \big( \Omega\times\mathbb{CP}^{1} \big)$ is a compact subset of the Hausdorff space $\Gamma_{u} \setminus \big( U \cap \mathrm{Q}_{2} \big)$, so $\Gamma_{u} \setminus \big( \Omega\times\mathbb{CP}^{1} \big)$ is closed, but its complement $ \Gamma_{u} \setminus \big( U \cap \mathrm{Q}_{2} \big) - \Gamma_{u} \setminus \big( \Omega\times\mathbb{CP}^{1} \big) $ has empty interior (because the classical limit set $\Lambda$ of $\Gamma$ has empty interior, see \cite{Ahlfors}), so this is a contradiction. Then, $U_{\Gamma}$ is a maximal open set where $\Gamma_{u}$ acts properly discontinuously; in particular, $U_{\Gamma} = \mathrm{Q}_{3}-\Lambda_{F}$. $\qquad\qquad\Box$. \\

We just found a family of complex orthogonal Kleinian groups, now we will see that some of these groups are a generalization of the (classical) Schottky groups.

We say that a finitely generated discrete subgroup $\Gamma$ of  $\mathrm{PO}(5,\mathbb{C})$  is a  \emph{complex  orthogonal Schottky group of genus $g$}  if there exists a collection  $\{C_{1}, D_{1}, \dots, C_{g}, D_{g} \}$ of open sets of   $\mathrm{Q}_{3},$  with disjoint closures, and a finite set  of generators $\{s_{1}, \dots, s_{g} \}$ of $\Gamma$, such that, for all  $i=1, \dots, g$,
$$ s_{i}(C_{i}^{c})=\overline{D_{i}}.$$

If the group $\Gamma \subset \mathrm{SL}(2,\mathbb{C})$, of  Theorem \ref{res2}, is a (clasical) Schottky group, then $\Gamma$ acts, by means of the homomorphism $\psi$, defined in (\ref{aut}), as a complex orthogonal Schottky group. The proof of this goes as follows: 

Consider the action of
  $\{I\} \times \mathrm{SU}(2) \subset \mathrm{SL}(2,\mathbb{C}) \times \mathrm{SL}(2,\mathbb{C})$ in $\mathrm{Q}_{3}$, by means of the homomorphism $\psi$, defined in (\ref{aut}), the quotient space
 $\big( \{I\} \times \mathrm{SU}(2) \big) \setminus \mathrm{Q}_{3}$ and the quotient map
\begin{eqnarray*}
\delta: \mathrm{Q}_{3} & \to & \big( \{I\} \times \mathrm{SU}(2) \big) \setminus \mathrm{Q}_{3}, \\
z & \mapsto & [z].
\end{eqnarray*}

Then, $\delta$ is continuous, open and $\mathrm{SL}(2,\mathbb{C})$-equivariant and  can be represented by
\begin{eqnarray*} \label{eqdelta}
\delta: \mathrm{Q}_{3} & \to & \mathbb{H}^{3} \cup \mathbb{CP}^{1},  \\
 \delta([z_{1}:z_{2}:z_{3}:z_{4}:z_{5} ]) & = &  \left\{ \begin{array}{ccc}
             & \rho \Bigg( \left( \begin{array}{cc}
            z_{1} & z_{2} \\
            z_{4} & z_{5}
           \end{array}
    \right)  \Bigg), & \ [z_{1}:z_{2}:z_{3}:z_{4}:z_{5} ] \in \Theta, \nonumber\\
            & \Bigg[ Im \left( \begin{array}{cc}
            z_{1} & z_{2} \\
            z_{4} & z_{5}
           \end{array}
    \right) \Bigg],  & \ [z_{1}:z_{2}:0:z_{4}:z_{5} ] \in  \mathrm{Q}_{2}, \nonumber\\
            \end{array}
                \right.
\end{eqnarray*}
where the function $\rho$ was defined in (\ref{defrho}).

Finally, we pull back the (classical) Schottky group to get a complex  ortogonal Schottky group of dimension three, that is, if    $\{C_{1}, D_{1}, \dots, C_{g}, D_{g} \}$ is a collection of open sets of   $\mathbb{CP}^{1} \cup \mathbb{H}^{3},$  with disjoint closure,  $\{s_{1}, \dots, s_{g} \}$ is a finite set  of generators of $\Gamma$, such that, for all  $i=1, \dots, g$,
$ s_{i}(C_{i}^{c})=\overline{D_{i}}.$ If for all $i=1, \dots,g$, you define $B_{i}:=s_{i}^{-1}(C_{i})$ and $E_{i}:=s_{i}^{-1}(D_{i})$, then $\{B_{1}, E_{1}, \dots, B_{g}, E_{g} \}$ is a collection of open sets of   $\mathrm{Q}_{3},$  with disjoint closure, $\{s_{1}, \dots, s_{g} \}$  is finite set  of generators of $\Gamma$ (considered as a complex orthogonal Kleinian group, that is, identified with $\psi(\Gamma)$), such that, for all  $i=1, \dots, g$,
$ s_{i}(B_{i}^{c})=\overline{E_{i}}.$  \\

As Shottky groups are stable under small perturbations, every small perturbation of $\Gamma$, inside $\mathrm{PO}(5,\mathbb{C})$, is also a complex orthogonal Schottky group of dimension three. In fact, it is not difficult to see that, all the corresponding quotients spaces are diffeomorphic to each other. \\

We would like to mention that the examples of complex  orthogonal Kleinian groups of Theorem \ref{res1}, as well as the last examples of complex orthogonal Shottky groups, 
make sense in the real case; so, we also give examples of Lorentzian Kleinian groups: 

Let us denote by   $\mathbb{R}^{2,n}$ the space  $\mathbb{R}^{n+2}$ endowed with the quadratic form $q^{2,n}=-x^{2}_{1}-x^{2}_{2}+x^{2}_{3}+ \dots +x^{2}_{n+2}.$ The isotropic cone of $q^{2,n}$ is the subset of  $\mathbb{R}^{2,n}$ on which  $q^{2,n}$ vanishes. We call $C^{2,n}$ this isotropic cone, with the origin removed. Let's denote by $\pi$ the projection from $\mathbb{R}^{2,n}$, minus the origin,  on  $\mathbb{RP}^{n+1}$. The set  $\pi(C^{2,n})$ is a smooth hypersurface  $\Sigma$ of $\mathbb{RP}^{n+1}$. Recall that this hypersurface turns out to be endowed with a natural
 Lorentzian conformal structure such that its group of conformal transformations is  $\mathrm{PO}(2,n)$ (see \cite[p. 886]{Frances}). We call the \emph{Einstein universe} this hypersurface $\Sigma$, together with this canonical conformal structure, and we denote it  by $\mathrm{Ein}_{n}$.

Let us suppose that   $\Gamma \subset \mathrm{SL}(2,\mathbb{R})$ is a torsion free, finitely generated (classical) Fuchsian group, with $\Omega$ as discontinuity domain in  $\mathbb{S}^{1}$ and $u:\Gamma \to \mathrm{SL}(2,\mathbb{R})$ is a group homomorphism. We have that
 \begin{equation} \label{gammaru}
 \Gamma^{r}_{u}:= \Big\{ \big(\gamma, u(\gamma) \big) : \gamma \in \Gamma \Big\} \subset \mathrm{SL}(2,\mathbb{R}) \times \mathrm{SL}(2,\mathbb{R})
 \end{equation}
is a subgroup of $\mathrm{PO}(2,2)$. If $\Gamma_{u}$
acts properly discontinuously in  $\mathrm{AdS}_{3}:=\mathrm{Ein}_{3}-\mathrm{Ein}_{2}$, then $\Gamma_{u}$ acts properly discontinuouly on 
\begin{equation}  \label{wgamma}
W_{\Gamma}:=\mathrm{AdS}_{3} \cup \big( \Omega\times\mathbb{S}^{1} \big). 
\end{equation}
Even more, if  $\Gamma\setminus\Omega$ is compact, then   $W_{\Gamma}$ is maximal. As the limit set of  $\Gamma$ in $\mathbb{S}^{1}$ has Lebesgue measure zero (see \cite{tukia}), the proof of this assertion is essentially the same as the proof of Theorem \ref{res1}. 
Also, if $\Gamma \subset \mathrm{SL}(2,\mathbb{R})$ is a a (classical) Shottky group, then $\Gamma$ is a Lorentzian Shottky group of dimension three (as those of Frances in \cite[p.23]{Frances}),  every small perturbation of $\Gamma$ inside $\mathrm{PO}(5,\mathbb{R})$ is also a Lorentzian Schottky group of dimension three and all the corresponding quotients spaces are diffeomorphic to each other.


\section{Quotients} \label{s5}
Let $\Gamma$ be a torsion free, convex-cocompact, (classical) Kleinian group and $u: \Gamma \to \mathrm{SL}(2, \mathbb{C})$  a group morphism sufficiently close to the constant morphism. Recall the definition of the group $\Gamma_{u}$ given in Theorem \ref{res2}.  Recall also that, in Section \ref{ss3}, we saw that $\Gamma_{u}$ is a torsion free subgroup of $\mathrm{PO}(5,\mathbb{C})$.

 By Propositions \ref{p1} and
 \ref{res1}, if $u:\Gamma\to\mathrm{SL}(2,\mathbb{C})$ is a group morphism sufficiently close to the constant morphism $I$, then the quotient $M(u,\Gamma)$ of the action of $\Gamma_{u}$ on $U_{\Gamma}$ is a complex manifold, where $U_{\Gamma}$ was defined in Proposition \ref{res1}.
  The manifold $\Gamma_{I} \setminus U_{I}$ is the Guillot manifold (see Section \ref{ss3} and \cite[p. 224, 225]{GuillotD}).
  In this Section, we will prove that, for every group morphism $u: \Gamma \to \mathrm{SL}(2, \mathbb{C})$ close enough to the constant morphism, the manifold $M(u,\Gamma)$   is indeed compact and diffeomorphic to the Guillot manifold; so, it defines an uniformizable complex orthogonal  structure of dimension three, on the Guillot manifold, that is close to the Guillot structure. 
  
  We will also use this result to construct some close uniformizable complex projective  structures on other compact complex manifold of dimension three. \\

\noindent \textbf{Proof of Theorem \ref{res2}:}
Let us consider a torsion-free, convex-cocompact, (classical) Kleinian subgroup $\Gamma$ of $\mathrm{SL}(2,\mathbb{C})$. By Propositions \ref{p1} and \ref{res1}, we know that there exists an open neighborhood $\mathcal{V}$ of the constant morphism, such that, if $u \in\mathcal{V}$, then $\Gamma_{u}$ acts properly discontinuously on $U_{\Gamma}$. We will first prove that for $u \in \mathcal{V}$,  $\Gamma_{u}$ acts uniformly on $U_{\Gamma}$; that is, the action
\begin{eqnarray} \label{aupd2}
\Gamma \times ( \mathcal{V} \times U_{\Gamma} ) & \to & \mathcal{V} \times U_{\Gamma} \nonumber\\
\big( \gamma,(u,x) \big) & \to & \Big( u, \big(\gamma, u(\gamma)\big)x \Big) \nonumber\\
\end{eqnarray}
is properly discontinuous. Then, we will consider a resolution $r:X \to \mathcal{V}$ of singularities of $\mathcal{V}$, and define an action of $\Gamma$ on $X \times U_{\Gamma}$, such that, the action on  $\{x\} \times U_{\Gamma}$ coincides with the action on $\{r(x)\} \times U_{\Gamma}$. Finally, we will see that  $\Gamma$  acts properly discontinuously on $X \times U_{\Gamma}$ and the quotient of the projection of $X \times U_{\Gamma}$ onto its first coordinate is a locally trivial fibration whose fibers are the manifolds $M(u,\Gamma)$. \\

Let us consider the restriction to $\mathrm{SL}(2,\mathbb{C})$ of any Euclidian norm in $\mathbb{C}^{4}$ and let us denote it by $\arrowvert\arrowvert \cdot \arrowvert\arrowvert$. As by definition $\Theta$ is biholomorphic to $\mathrm{SL}(2,\mathbb{C})$, $\arrowvert\arrowvert \cdot \arrowvert\arrowvert$ defines a norm in $\Theta$; we will also denote it by $\arrowvert\arrowvert \cdot \arrowvert\arrowvert$.

Recall that $\Theta$ is biholomorphic 
Let us consider a finite set of generators
 $\mathcal{G}$  of $\Gamma$, any norm    $|| \cdot ||$ in $\Theta=\mathrm{Q}_{3}-\mathrm{Q}_{2}$ and  for all $\delta>0$, the neighborhood
 $$V_{\delta}(I):=\Big\{ g \in \mathrm{Hom} \big( \Gamma,\mathrm{SL}(2,\mathbb{C}) \big):\forall s \in \mathcal{G}, \ ||g(s)-I||  \le \delta \Big\},$$
  of the constant morphism, with respect to the compact-open topology, where $I$ is the identity matrix.\\

 We will prove that there exists $\epsilon>0$, such that, for every convergent sequence $(u_{n},z_{n})$ in $V_{\epsilon}(I) \times U_{\Gamma}$ and for every divergent sequence $( g_{n} )$ of $\Gamma$, such that $( g_{n},u_{n}(g_{n}))$ diverges simply to infinity, then $\big( ( g_{n},u_{n}(g_{n}) ) z_{n} \big)$ does not converge in $U_{\Gamma}$. This will imply that, if $\mathcal{V}:=V_{\epsilon}(I)$, then (\ref{aupd2}) is properly discontinuous. Note that, by Proposition \ref{res1}, this is true for constant sequences $(u_{n})$.

Recall the $\mathrm{SO}(4,\mathbb{C})$-equivariant biholomorphism  between $\Theta$ and $\mathrm{SL}(2,\mathbb{C})$. By Proposition \ref{p1}, there exists $\epsilon>0$,  such that, for all convergent sequences $\big( u_{n},z_{n} \big)$ in
$V_{\epsilon}(I) \times \mathrm{SL}(2,\mathbb{C})$, and for all divergent sequences $(g_{n})$, the sequence
$\Big( \big( g_{n}, u_{n}(g_{n}) \big) z_{n} \Big)$ is not convergent in $\Theta$. \\

Suppose $(u_{n})$ is in $V_{\epsilon}(I)$, $(g_{n})$ is divergent in $\Gamma$ and $\big( g_{n}, u_{n}(g_{n}) \big)$ diverges simply to infinity in  $\mathrm{PO}(5,\mathbb{C})$.
Then, by the last paragraph, there are no points in $\Theta$ dynamically related to each other which correspond to the sequence $\big( g_{n}, u_{n}(g_{n}) \big)$. 

As sequences of bounded distortion have points dynamically related in $\Theta$ (see Proposition \ref{sda}) then, $\big( g_{n}, u_{n}(g_{n}) \big)$  is of balanced or mixed distortion and, by Propositions \ref{sdb} and \ref{sdm}, it has associated two limit light geodesics contained in $\mathrm{Q}_{2}$, one attractor $\Delta^{+}$ and other repeller $\Delta^{-}$.

We claim that $\Delta^{+}$ and $\Delta^{-}$ are vertical and contained in $\Lambda\times\mathbb{CP}^{1}$. The proof of this is essentially the same  (see Proposition \ref{res1}) as the proof   that the limit light geodesics of  $\Gamma_{u}$ are vertical and contained in $\Lambda\times\mathbb{CP}^{1}$ (it is only necessary to change the sequence $\big( g_{n}, u( g_{n}) \big)$ by the sequence $\big( g_{n}, u_{n}( g_{n}) \big)$ and use Proposition \ref{l0}).
Then, for each convergent sequence $(z_{n})$ of $U_{\Gamma}$, the sequence $\big( \big( g_{n}, u_{n}(g_{n}) \big) z_{n} \big)$ does not converge in  $U_{\Gamma}$. Then, (\ref{aupd2}) is properly discontinuous.
\\

Now, note that the projection of $\mathcal{V} \times U_{\Gamma}$ onto its first factor, defines the continuous function
\begin{eqnarray*}
t: \Gamma \setminus \big( \mathcal{V} \times U_{\Gamma} \big) & \to &  \mathcal{V}, \\
\ [(u,y)] & \mapsto u.
\end{eqnarray*}

If the algebraic variety $\mathcal{V}$ has no singularities in a neighborhood of the identity, i.e., it is a manifold at the identity (as, for example, when $\Gamma$ is a classical Schottky group), then, by the Ehresmann Fibration Lemma, $t$ is a locally trivial fibration and so all $M(u,\Gamma)$ are diffeomorphic to each other, for $u$ sufficiently close to the constant morphism. However, $\mathcal{V}$ can have singularities  arbitrarily close to the identity (see, for example, the Fuchsian groups of \cite[p.\,567]{GoldmanCC}); if this is the case, consider a resolution of the singularities (see \cite{Hironaka}) of $\mathcal{V}$, that is, let $r: X \to \mathcal{V}$ be an holomorphic surjective proper map, where $X$ is a complex manifold. Let us consider $X \times U_{\Gamma}$ as a differential manifold and define the action
\begin{eqnarray} \label{acgxxu}
\Gamma \times \big( X \times U_{\Gamma} \big)  & \to &   X \times U_{\Gamma},  \nonumber\\
\Big( \gamma, \big( x,y \big) \Big) & \mapsto & \Big( x, \big( \gamma ,   r(x) (\gamma) \big) (y) \Big),
\end{eqnarray}
 and  the map
\begin{eqnarray*}
r \times I:X \times U_{\Gamma} & \to & \mathcal{V}  \times U_{\Gamma}, \\
(x,y) & \mapsto & \big( r(x), y \big),
\end{eqnarray*}
which is $\Gamma$-equivariant:
$$\big( r \times I \big) \big( \gamma, (x,y) \big)= \Big( r(x), \big( \gamma ,   r(x) (\gamma) \big) (y) \Big)= \Big( \gamma, \big( r(x),y \big) \Big)=\big( \gamma, (r \times I) ( x,y ) \big).
$$

As $\Gamma$ acts properly discontinuously on $\mathcal{V} \times U_{\Gamma}$, then it acts properly discontinuously on $X \times U_{\Gamma}$. As the projection of $X \times U_{\Gamma}$ onto its first factor is $\Gamma$-invariant, then it defines a submersion
\begin{eqnarray*}
s: \Gamma \setminus \Big( X \times U_{\Gamma} \Big) & \to &  X,  \\
\ [(x,u)] & \mapsto x.
\end{eqnarray*}

Let us consider any $y \in r^{-1}(I)$, then as $s^{-1}(y)$ was defined to be the Guillot manifold $t^{-1}(I)$, then it is compact. So, by the Ehreshmann Fibration Lemma, $s$ is trivial bundle in a neighborhood $\mathcal{U}$ of $y$. As every fiber $s^{-1}(u)$, of $s$ over $u \in \mathcal{U}$,  was defined to be as the fiber $t^{-1} \big( r(u) \big)$, of $t$ over $r(u)$, and as $r$ is open,  then $r(\mathcal{U})$ is an open neighborhood of $I$ such that for all $v \in  r(\mathcal{U})$, $t^{-1}(v)$ is compact and diffeomorphic to the Guillot manifold. $\Box$ \\

Let us consider the geometry  $\big(\mathrm{Q}_{3}, \mathrm{PO}(5,\mathbb{C}) \big)$, the intermediate covering  $U_{\Gamma}$ of the Guillot manifold  $\Gamma\setminus U_{\Gamma}$ and the developing map
\begin{eqnarray*}
D: U_{\Gamma} & \to & \mathrm{Q}_{3}, \\
x & \mapsto & x,
\end{eqnarray*}
defined in this intermediate covering. Let us suppose that the hypothesis of the last Theorem are satisfied, for each  homomorphism
 $u: \Gamma \to \mathrm{SL}(2,\mathbb{C})$, close enough to the constant morphism,  consider the representation
\begin{eqnarray*}
\rho_{u}: \Gamma & \to &  \mathrm{SO}(4,\mathbb{C}) \subset  \mathrm{PO}(5,\mathbb{C}), \\
\gamma & \to & \big( \gamma,  u(\gamma) \big),
\end{eqnarray*}
 of $\Gamma$ induced by $u$.

Let us consider the deformation space of  $\big(\mathrm{Q}_{3}, \mathrm{PO}(5,\mathbb{C}) \big)$-structures on the Guillot manifold (see \cite[p. 13]{GoldmanGE}), then
   $( D, \rho_{u})$ determines an unique uniform complex orthogonal  structure of dimension three on the  Guillot manifold which is close to the Guillot structure. Also, if
 $v: \Gamma \to \mathrm{SL}(2,\mathbb{C})$ is another group morphism close enough to the constant morphism, then  $\rho_{u}$ is conjugate to $\rho_{V}$ if and only if $u$ is conjugate to $v$. Then, the complex  orthogonal  structures determined by  $(D,\rho_{u})$ and $(D,\rho_{v})$ are the same if and only if  $u $ and $v$ are conjugate. \\

It should be pointed out that there are examples of (classical) Kleinian groups for which there is a family of such mutually non-conjugate group morphisms from $\Gamma$ to $\mathrm{SL}(2,\mathbb{C})$:

 If $\Gamma$ is a (classical)  Schottky group of genus  $g$, then the space  $\mathrm{Hom}(\Gamma, \mathrm{SL}(2,\mathbb{C}))$ of group morphisms from $\Gamma$ to $\mathrm{SL}(2,\mathbb{C})$ is a complex manifold, of dimension $3g$, biholomorphic to $\mathrm{SL}(2, \mathbb{C})^{g}$.
Since we have two free 3-dimensional parameters and the conjugation is by a family of dimension three, there exists, even up to conjugation, an infinite family of homomorphisms arbitrarily close to the constant morphism.

If $\Gamma$ is a (classical) Fuchsian group, then, as any matrix commute with itself, for every pair of matrices $A$ and $C$ in $\mathrm{SL}(2,\mathbb{C})$, the equation $[A,A][C,C]=I$ is satisfyed; then, by \cite[p. 567]{GoldmanCC}, $\Gamma$ is a Fuchsian group. Since we have, at least, two free 3-dimensional parameters and the conjugation is by a family of dimension three, there exists, even up to conjugation, an infinite family of homomorphisms arbitrarily close to the constant morphism. \\

We would like to mention that, in the real case, there is a weaker version of  the last Theorem; in particular, if $\Gamma\subset \mathrm{SL}(2,\mathbb{R})$ is a torsion free, convex-cocompact (classical) Kleinian group with domain of discontinuity $\Omega$ in $\mathbb{S}^{1}$; then, for all group morphisms $u: \Gamma \to \mathrm{SL}(2,\mathbb{R})$ sufficiently close to the identity morphism, then $\Gamma^{r}_{u}$, which was defined in (\ref{gammaru}), acts properly discontinuously on $W_{\Gamma}$, defined in (\ref{wgamma}).

Now, we will consider some quotients of the manifolds $M(u,\Gamma)$ of Theorem  \ref{res2} to construct uniformizable complex projective  structures locally modeled on ($\mathbb{CP}^{4}, \mathrm{PGL}(4, \mathbb{C})$) on a related compact complex manifold of dimension three. \\

\noindent\textbf{Proof of Theorem \ref{corart}:} \\
Let $\Gamma \subset \mathrm{SL}(2,\mathbb{C})$ be a torsion free, convex-cocompact, (classical) Kleinian group with domain of discontinuity $\Omega$ in $\mathbb{CP}^{1}$. We will prove that  there exists a $\Gamma_{u}$-equivariant continuous, open and proper function $f$ from
  $\mathrm{Q}_{3}$ onto  $\mathbb{CP}^{3}$. By Theorem \ref{res2}, this will imply that if $u$ is a group morphism sufficiently close to the constant morphism, then $\Gamma_{u}$ acts properly discontinuously on $V_{\Gamma}:=f(U_{\Gamma})$ and the quotient is compact. As $I \times f$ is proper, then $\Gamma$ acts properly discontinuously on $\mathcal{V} \times V_{\Gamma}$, where this action is   the action of $\Gamma_{u}$ on $V_{\Gamma}$ in the fibers $\{u\} \times V_{\Gamma}$. Finally, we will consider a resolution of of singularities $r:X \to \mathcal{V}$ and follow the same reasoning of Theorem \ref{res2} to prove that, for all $u \in\mathcal{V}$, all the quotients $\Gamma_{u} \setminus V_{\Gamma}$ are diffeomorphic to each other.\\

 Let us consider the involution
\begin{eqnarray*}
 j: \mathrm{Q}_{3} & \to & \mathrm{Q}_{3}, \\
 \ [z_{1}:z_{2}:z_{3}:z_{4}:z_{5} ] & \mapsto & [-z_{1}:-z_{2}:z_{3}:-z_{4}:-z_{5} ],
\end{eqnarray*}
which generates the subgroup $\Sigma:=\langle j \rangle = \{ j, I \}$ of $\mathrm{O}(4,\mathbb{C})$. The composition of the quotient map $Q_{3} \to \Sigma\setminus \mathrm{Q}_{3}$, defined by the action of $\Sigma$ on $\mathrm{Q}_{3}$, and the biholomorphism
\begin{eqnarray*}
 \Sigma \setminus \mathrm{Q}_{3} & \to & \mathbb{CP}^{3}, \\
 \big| [z_{1}:z_{2}:z_{3}:z_{4}:z_{5} ] \big| & \mapsto & [z_{1}:z_{2}:z_{4}:z_{5} ],
\end{eqnarray*}
is a continuous and proper function $f$, where $|z|$ denotes the equivalence class of $z$. Note that $f$ restricted to $\Theta$ is the usual $2$ to $1$ covering of $\mathrm{PSL}(2,\mathbb{C})$ by $\mathrm{SL}(2,\mathrm{C})$ and restricted to $\mathrm{Q}_{3}$ is the identity.
 \\

 The action of $\big( \mathrm{SL}(2,\mathbb{C}) \times \mathrm{SL}(2,\mathbb{C}) \big)$ on $\mathrm{Q}_{3}$ induces, in a unique way, an action of $\mathrm{SL}(2,\mathbb{C}) \times \mathrm{SL}(2,\mathbb{C})$   on $\mathbb{CP}^{3}$ such that $f$ is $\big(\mathrm{SL}(2,\mathbb{C}) \times \mathrm{SL}(2,\mathbb{C})\big)$-equivariant. Even though this action is not faithful, it defines an holomorphic homomorphism $\tau$ from $\mathrm{SL}(2,\mathbb{C}) \times \mathrm{SL}(2,\mathbb{C})$ to $\mathrm{PO}(4,\mathbb{C})$, whose image is the projectivization $\mathrm{PSO}(4,\mathbb{C})$ of the orthogonal matrices of determinant one, and whose kernel is $\big\{ (I,I), (-I,I),(I,-I),(-I,-I) \big\}$. As $-I \notin \Gamma$, where $I$ is the identity matrix, then the restriction of $\tau$ to $\Gamma_{u}$ is a monomorphism. As usual, we identify the domain and the image. 
 
 As $\Gamma_{u}$ acts properly discontinuously on $U_{\Gamma}$, it also acts properly discontinuously on $V_{\Gamma}:=f(U_{\Gamma})$. So, $\Gamma_{u} \setminus V_{\Gamma}$ is a complex manifold and $f$ defines a continuous function on the quotient
\begin{eqnarray*}
  \Gamma_{u} \setminus U_{\Gamma} & \to &  \Gamma_{u} \setminus V_{\Gamma}, \\
  \ [w] & \mapsto & [f(w) ].
\end{eqnarray*}
In fact, the quotient  $\Gamma_{u} \setminus f(U_{\Gamma})$ is a quotient of the manifold  $\Gamma_{u} \setminus U_{\Gamma}$: For all morphism $u: \Gamma \to  \mathrm{SL}(2,\mathbb{C})$, the involution
 $j$ commutes with the action of  $\Gamma_{u}$ in  $\mathrm{Q}_{3}$ and then, it is well defined in the quotient  $\Gamma_{u} \setminus U_{\Gamma}$, that is, defines a transformation $\widetilde{j}:\Gamma_{u} \setminus \mathrm{Q}_{3} \to \Gamma_{u}\setminus \mathrm{Q}_{3}, \ \arrowvert [z_{1}:z_{2}:z_{3}:z_{4}:z_{5}] \arrowvert \mapsto   \arrowvert [-z_{1}:-z_{2}:z_{3}:-z_{4}:-z_{5}] \arrowvert $. The group $\widetilde{\Sigma}$ generated by $\widetilde{j}$ acts freely and properly discontinuously on   $\Gamma_{u} \setminus U_{\Gamma}$ and, as the actions of
 $\Sigma$ and $\Gamma_{u}$ comute, then, the quotients  $\widetilde{\Sigma} \setminus (\Gamma_{u} \setminus U_{\Gamma} )$ and $\Gamma_{u} \setminus f(U_{\Gamma})$ are biholomorphic. \\

It also happens that $V_{\Gamma}$ is maximal; otherwise, we could pull back an invariant open set $V \supset V_{\Gamma}$, where $\Gamma_{u}$ acts properly discontinuously and contradict the maximality of $U_{\Gamma}$. \\

By Theorem \ref{res2}, there exists an open neighborhood $\mathcal{V}$ of the constant morphism, such that, if  the homomorphism $u:\Gamma \to \mathrm{SL}(2,\mathbb{C})$ belongs to it, then $\Gamma_{u} \setminus U_{\Gamma}$ is compact, so $\Gamma_{u} \setminus V_{\Gamma}$ is also compact. 

Now, let us define the following action
\begin{eqnarray*}
  \Gamma\times\big( \mathcal{V} \times V_{\Gamma} \big) & \to & \big( \mathcal{V} \times V_{\Gamma} \big),  \\
  \big( \gamma, (v,x)  \big) & \mapsto & \Big(v, \big(\gamma,v(\gamma) \big) x \Big).
\end{eqnarray*}

As $I \times f$ is proper, the last action is properly discontinuous. Finally, let us consider a resolution of singularities $r:X \to \mathcal{V}$, where $X$ is a complex manifold and $r$ is holomorphic, then, by the same reasoning of Theorem \ref{res2}, for all $u \in\mathcal{V}$, all the quotients $\Gamma_{u} \setminus V_{\Gamma}$ are diffeomorphic to each other. $\Box$ \\

Let us suppose that the hypothesis of the last Theorem are satisfied. By the same arguments of Section \ref{s4}, but applied to $\mathbb{CP}^{3}$, instead of $\mathrm{Q}_{3}$, if follows that the vertical light geodesics of $\Lambda \times \mathbb{CP}^{1}$ are also attractor and repeller limit light geodesics for the action of $\Gamma_{u}$ on $V_{\Gamma}$. \\

 Let us
 consider the geometry  $\big(\mathbb{CP}^{3}, \mathrm{PGL}(4,\mathbb{C}) \big)$. We denote by $\mathcal{M}$ the differentiable manifold defined by the quotients $\Gamma_{u}\setminus V_{\Gamma}$ of last Theorem. This manifold was discovered by Guillot. Also, Guillot found the uniformizable complex projective  structure induced by $\Gamma_{I}\setminus V_{\Gamma}$. As we said in the Introduction, we call the $G$ manifold the quotient manifold (both differentiable and complex) and the $G$ structure the complex projective structure determined by it.

  Let us  consider the intermediate covering $V_{\Gamma}$ of  $\mathcal{M}$ and the developing map
\begin{eqnarray*}
D: V_{\Gamma} & \to & \mathrm{Q}_{3}, \\
x & \mapsto & x,
\end{eqnarray*}
defined in the intermediate covering. Let us suppose that the hypothesis of the last Theorem are satisfied, for each  group morphism
 $u: \Gamma \to \mathrm{SL}(2,\mathbb{C})$ close enough to the constant morphism,  consider the representation
\begin{eqnarray*}
\rho_{u}: \Gamma & \to &  \mathrm{SO}(4,\mathbb{C}) \subset  \mathrm{PO}(5,\mathbb{C}), \\
\gamma & \to & \big( \gamma,  u(\gamma) \big),
\end{eqnarray*}
 of $\Gamma$ induced by $u$.

Let us consider the deformation space of  $\big(\mathbb{CP}^{3}, \mathrm{PGL}(4,\mathbb{C}) \big)$-structures on  $\mathcal{M}$ (see \cite[p. 13]{GoldmanGE}), then
   $( D, \rho_{u})$ determines an unique uniformizable complex projective  structure on $\mathcal{M}$. Also, if
 $v: \Gamma \to \mathrm{SL}(2,\mathbb{C})$ is also a group morphism close enough to the constant morphism, then  $\rho_{u}$ is conjugated to $\rho_{V}$ if and only if $u$ is conjugated to $v$. Then, the complex projective  structures determined by  $(D,\rho_{u})$ and $(D,\rho_{v})$ are the same if and only if  $u $ and $v$ are conjugated. \\
 
 It should be pointed out that, in this case, we also have that Schottky and Fuchsian groups are examples of (classical) Kleinian groups for which there is a family of such mutually non-conjugated homomorphisms from $\Gamma$ to $\mathrm{SL}(2,\mathbb{C})$.

\bibliography{refs3}
\bibliographystyle{amsplain}
 \end{document}